\documentclass[hidelinks]{alt2026} %
\usepackage[margin=3cm]{geometry}

\usepackage[T1]{fontenc}    %
\usepackage{microtype}      %
\usepackage{lmodern}

\usepackage[normalem]{ulem}

\usepackage{xcolor}

\usepackage{mathrsfs}
\usepackage{mathtools,amssymb}   %
\mathtoolsset{showonlyrefs,showmanualtags}

\usepackage{bm}             %

\usepackage{booktabs}

\usepackage{my_notation}

\usepackage{amsthm}
\usepackage{thmtools}                   %

\usepackage{my_citations}   %
\usepackage{csquotes}
\usepackage[british]{babel}     %

\usepackage{hyperref}       %
\usepackage{bookmark}           %
\usepackage{zref-clever}
\let\cref\zcref
\zcsetup{cap,nameinlink=false}                   %

\AddToHook{env/algorithm/begin}{%
   \zcsetup{reftype=algorithm}}

\zcRefTypeSetup{algorithm}{
Name-sg = Algorithm,
name-sg = algorithm,
Name-pl = Algorithms,
name-pl = algorithms,
}

\usepackage{my_theorems}

\usepackage{enumitem}           %
\usepackage{xparse}

\title{Maximum effective dimension and information gain}

\altauthor{%
 \Name{David Janz} \Email{david.janz@stats.ox.ac.uk}\\
 \addr University of Oxford%
 \AND
 \Name{Arya Akhavan} \Email{arya.akhavan@stats.ox.ac.uk }\\
 \addr University of Oxford
 \AND
  \Name{Alexandre B. Tsybakov} \Email{alexandre.tsybakov@ensae.fr }\\
 \addr CREST, ENSAE, Institut Polytechnique de Paris
}

\begin{document}

\maketitle

\begin{abstract}
  We establish general upper bounds on the maximum effective dimension and information gain of kernel Gram matrices over arbitrary designs in terms of approximation properties of the kernel or, equivalently, of the corresponding reproducing kernel Hilbert space. We show that across three regularity regimes, covering as special cases the Mat\'ern and squared exponential kernels, these bounds cannot be improved beyond constant factors.
\end{abstract}

\section{Introduction}

Let \(k\) be a positive semidefinite kernel on a set \(\cX\), let \(X_n=(x_1,\dots,x_n)\in\cX^n\), and write
\[
  G_k(X_n):=\squareb{k(x_i,x_j)}_{i,j=1}^n
\]
for the Gram matrix of \(k\) on \(X_n\).

The effective dimension associated with \(G_k(X_n)\) at scale \(\rho>0\) \citep{zhang2005learning}, also known as the fixed-design degrees of freedom \citep{craven1979smoothing}, is
\[
  \effd^k(X_n;\rho):=\tr\roundb{G_k(X_n)\roundb{G_k(X_n)+\rho I}^{-1}}.
\]
This provides a soft count of the number of eigenvalues of \(G_k(X_n)\) at scale \(\rho\) or larger.

Information gain, which arose in Bayesian experimental design \citep{lindley1956measure} and in the form considered here in the Gaussian process literature \citep{seeger2008information,srinivas2009gaussian}, is
\[
  \gamma^k(X_n;\rho):=\log\det\roundb{I+\rho^{-1}G_k(X_n)}.
\]
Under a Gaussian process model with independent Gaussian observation noise of variance \(\rho\), one half of this quantity is the mutual information between the latent values at \(X_n\) and their noisy observations \citep{rasmussen2006gaussian,seeger2008information}.

We study the maximum growth of both quantities over the design $X_n$. For $n\in\Np$ and $\rho>0$, define
\[
  \Effd^k_n(\cX;\rho):=\sup_{X_n\in\cX^n}\effd^k(X_n;\rho), \qquad \Gamma^k_n(\cX;\rho):=\sup_{X_n\in\cX^n}\gamma^k(X_n;\rho).
\]
These suprema provide distribution-free bounds on the empirical quantities, including when the design is chosen adversarially or sequentially \citep{abbasi2013online,srinivas2009gaussian}. Our contributions are:
\begin{enumerate}
  \item We establish a general approximation-theoretic upper-bound principle for the maximum effective dimension, yielding bounds across three regularity regimes: finite smoothness, ultradifferentiable classes, and finite-order entire extensions.  Corresponding maximum information gain bounds follow from an integral identity.
  \item We show that the upper bounds are sharp in each regularity regime by identifying corresponding stationary kernel classes for which matching lower bounds hold. These classes are characterised by algebraic, stretched-exponential, and super-exponential lower bounds on their spectral densities.
\end{enumerate}
Prior work on the quantities considered here has focused primarily on maximum information gain for the Mat\'ern and squared exponential kernels \citep{srinivas2009gaussian,iwazaki2025improved,iwazaki2026tighter}. For these two special kernels, our general construction implies upper and lower bounds matching up to constant factors, closing the existing gaps in the known results. Previous bounds on the maximum effective dimension were typically recovered from information-gain bounds, whereas our upper-bound analysis proceeds in the opposite direction. Information-gain upper bounds based on Mercer eigenvalue decay were given by \citet{vakili2021information} under an assumption of uniformly bounded Mercer eigenfunctions but, to the best of our knowledge, this assumption has not been verified for either kernel. A detailed comparison with previous results is given in \cref{sec:related-work}.

Effective dimension and information gain are pivotal in several problems of statistics and
learning theory. Effective dimension plays a key role in fixed-design risk bounds for kernel regression \citep{zhang2005learning,bach2013sharp}. At the population level, the corresponding complexity is encoded by capacity assumptions on the decay of $\tr\roundb{T(T+\rho I)^{-1}}$, where $T$ is the kernel integral operator associated with a fixed covariate distribution \citep{caponnetto2007optimal}. Effective dimension also controls the number of columns required for Nystr\"om approximations of the gram matrix \citep{bach2013sharp,elalaoui2015fast,musco2017recursive} and the number of features required for accurate random Fourier feature approximations \citep{avron2017random}.

Information gain is the regularised kernel analogue of the log-determinant criterion in $D$-optimal design \citep{kiefer1959optimum} and appears in worst-case regret bounds for online prediction with Gaussian process models \citep{kakade2005worst}. More broadly, it enters regret and sample-complexity bounds for kernel bandits, Bayesian optimisation, reinforcement learning, and best-arm identification \citep{srinivas2009gaussian,shekhar2018gaussian,janz2020bandit,kassraie2022neural,vakili2023kernelized,salgia2021domain,li2022gaussian,vakili2021optimal}. Two remarkable properties can be mentioned in this respect. First, information gain determines predictive confidence widths for kernel ridge regression with sequentially selected covariates \citep{abbasi2013online} and one can show that this dependence cannot be removed in general \citep{lattimore2023lower}. Second, the sum of the online kernel ridge regression predictive variances scales as twice the information gain, cf. the elliptical-potential lemma  \citep[Theorem~11.7]{cesa2006prediction}.

\paragraph{Notation} All kernels are real-valued and positive semidefinite. For kernels $k,K$, write $k\succeq K$ if $k-K$ is positive semidefinite; the same notation is used for Hermitian matrices. For nonnegative functions or sequences $a,b$ depending on common arguments, write $a\lesssim b$ if $a\leq Cb$ pointwise for some constant $C>0$ independent of those arguments; write $a\gtrsim b$ if $b\lesssim a$, and $a\asymp b$ if both $a\lesssim b$ and $b\lesssim a$. Generic constants $c,C>0$ may change from line to line, with dependence on fixed parameters indicated by subscripts when needed. We write $\N$ for the nonnegative integers, $\Np$ for the positive integers, and $[m]:=\{1,\ldots,m\}$. We write $\norm{\cdot}$ for the Euclidean norm, $\langle x\rangle:=\roundb{1+\norm{x}^2}^{1/2}$, and $u\wedge v:=\min\{u,v\}$. For $\alpha=(\alpha_1,\ldots,\alpha_q)\in\N^q$, write
\[
  |\alpha|:=\sum_{j=1}^q\alpha_j,
  \qquad
  \alpha!:=\prod_{j=1}^q\alpha_j!,
  \qquad
  \partial^\alpha:=\partial_1^{\alpha_1}\cdots\partial_q^{\alpha_q}.
\]

\section{Main results}
Our effective-dimension upper bounds follow by majorising the original kernel with the sum of two kernels, one of which is low rank and the other small on the diagonal.

\begin{theorem}[name={},restate=ThmFiniteRank]\label{thm:abstract-approximation-principle}
  Let \(\cX \subset \Rd\) be a set and let \(k\) be a positive semidefinite kernel on \(\cX\).
  Suppose that for every \(M\in\Np\) there exist positive semidefinite kernels
  \[
    K_M,\ R_M \colon \cX\times \cX \to \R
  \]
  and \(\epsilon_M > 0\) such that the following hold.
  \begin{enumerate}
    \item For every \(n\in\Np\) and \(X_n \in \cX^n\),
    \[
      G_k(X_n)
      \preceq
      G_{K_M}(X_n) + G_{R_M}(X_n).
    \]
    \item For every \(n\in\Np\) and \(X_n \in \cX^n\), $\rank G_{K_M}(X_n) \lesssim M^d$.
    \item For every \(x \in \cX\), $R_M(x,x) \leq \epsilon_M^2$.
  \end{enumerate}
  Then, for every \(n\in\Np\), \(\rho>0\), and \(X_n\in\cX^n\), we have
  \[
    \tr\roundb[\big]{
      G_k(X_n)\roundb{G_k(X_n)+\rho I}^{-1}
    }
    \lesssim
    \inf_{M\in\Np}
    \curlyb[\Big]{M^d + \frac{n}{\rho}\epsilon_M^2}.
  \]
\end{theorem}

\begin{proof}
  Fix \(M\in\Np\), \(n\in\Np\), \(\rho>0\), and
  \(X_n=(x_1,\dots,x_n)\in\cX^n\). Write
  \[
    G:=G_k(X_n),
    \qquad
    A:=G_{K_M}(X_n),
    \qquad
    B:=G_{R_M}(X_n),
  \]
  and set $F_\rho(X):=X(X+\rho I)^{-1}$. Since $F_\rho(X)=I-\rho(X+\rho I)^{-1}$, the map \(F_\rho\) is monotone on the positive semidefinite cone.
  Hence, by \(G\preceq A+B\), $\tr F_\rho(G) \leq \tr F_\rho(A+B)$.

  Let \(\Pi\) be the orthogonal projection onto
  \(\operatorname{ran}(A)\), and put \(Q:=I-\Pi\). Since
  \[
    0\preceq F_\rho(A+B)\preceq I,\qquad
    F_\rho(A+B)\preceq \rho^{-1}(A+B),\qquad QAQ=0,
  \]
  we have
  \begin{align*}
    \tr F_\rho(A+B)
    =\tr\roundb{\Pi F_\rho(A+B)\Pi} + \tr\roundb{Q F_\rho(A+B)Q}
    &\leq \tr\Pi + \rho^{-1}\tr\roundb{Q(A+B)Q} \\
    &= \rank A + \rho^{-1}\tr(QBQ) \\
    &\leq \rank A + \rho^{-1}\tr B.
  \end{align*}
  By the hypotheses, $\rank A\lesssim M^d$ and $\tr B = \sum_{i=1}^n R_M(x_i,x_i) \leq n\epsilon_M^2$. Consequently, $\tr\roundb{G(G+\rho I)^{-1}}
    \lesssim
    M^d+\frac{n}{\rho}\epsilon_M^2$. Taking the infimum over \(M\in\Np\) proves the theorem.
\end{proof}

An exact integral identity then converts effective-dimension bounds into information-gain bounds.

\begin{lemma}
\label{lem:information-gain-from-effective-dimension}
  Let $k$ be a positive semidefinite kernel on $\cX$ with $\kappa_0:=\sup_{x\in\cX}k(x,x)<\infty$. Let $\Phi:(0,\infty)\to[0,\infty)$ be measurable, and suppose that $\Effd_n^k(\cX;\rho) \leq \Phi(n/\rho)$ for every $n\in\Np$ and $\rho>0$. Then, for every $n\in\Np$ and $\rho>0$,
  \[
    \Gamma_n^k(\cX;\rho)
    \leq
    \int_0^{n/\rho}
    \roundb[\Big]{\kappa_0 b\wedge\Phi(b)}
    \frac{\dif b}{b}.
  \]
\end{lemma}

\begin{proof}
  Fix $X_n\in\cX^n$, write $G:=G_k(X_n)$, and let
  $\lambda_1,\dots,\lambda_n\geq0$ be its eigenvalues. For every $\lambda\geq0$,
  \[
    \int_\rho^\infty
    \frac{\lambda}{\lambda+t}\frac{\dif t}{t}
    =
    \int_\rho^\infty
    \frac{1}{t}-\frac{1}{\lambda+t}\dif t
    =
    \log\roundb{1+\lambda/\rho}.
  \]
  Summing over the eigenvalues gives the integral identity \[\gamma^k(X_n;\rho)
    =
    \int_\rho^\infty
    \effd^k(X_n;t)\frac{\dif t}{t}.\] Moreover, $G(G+tI)^{-1}\preceq t^{-1}G$ and
  $\tr G\leq\kappa_0n$, so
  \[
    \effd^k(X_n;t)
    \leq
    \frac{\kappa_0n}{t}
    \wedge
    \Effd_n^k(\cX;t)
    \leq
    \frac{\kappa_0n}{t}
    \wedge
    \Phi(n/t).
  \]
  Therefore, changing variables $b=n/t$,
  \[
    \gamma^k(X_n;\rho)
    \leq
    \int_\rho^\infty
    \roundb[\Big]{\frac{\kappa_0n}{t}\wedge\Phi(n/t)}
    \frac{\dif t}{t}
    =
    \int_0^{n/\rho}
    \roundb[\Big]{\kappa_0b\wedge\Phi(b)}
    \frac{\dif b}{b}.
  \]
  Taking the supremum over $X_n\in\cX^n$ proves the claim.
\end{proof}

The preceding results reduce the upper bounds to the decay of the residual sequence $(\epsilon_M)_{M\in\Np}$. We record the rates corresponding to the three residual profiles used below.

\begin{corollary}[name={},restate=CorCanonicalProfiles]
\label{cor:canonical-residual-profiles}
  Suppose that the hypotheses of
  \cref{thm:abstract-approximation-principle}
  hold with residual sequence \((\epsilon_M)_{M\in\Np}\), and suppose that $\kappa_0
    :=
    \sup_{x\in\cX} k(x,x)
    <
    \infty$. For \(n\in\Np\) and \(\rho>0\), let $a:=n/\rho$ be the effective sample size. Then the following hold.

  \begin{enumerate}
    \item \emph{Algebraic residuals.}
    Suppose that, for some \(s>0\), $\epsilon_M
      \lesssim M^{-s}$. Then
    \[
      \Effd_n^k(\cX;\rho)
      \lesssim
      a^{d/(d+2s)}
      \spaced{and}
      \Gamma_n^k(\cX;\rho)
      \lesssim
      a^{d/(d+2s)}.
    \]

    \item \emph{Stretched-exponential residuals.}
    Suppose that, for some \(\sigma\geq1\) and \(c>0\), $\epsilon_M
      \lesssim
      \exp\curlyb{-cM^{1/\sigma}}$. Then
    \[
      \Effd_n^k(\cX;\rho)
      \lesssim
      \roundb{\log(e+a)}^{d\sigma} \spaced{and} \Gamma_n^k(\cX;\rho)
      \lesssim
      \roundb{\log(e+a)}^{d\sigma+1}.
    \]

    \item \emph{Factorial residuals.}
    Suppose that, for some \(c>0\), $\epsilon_M
      \lesssim
      \exp\curlyb{-cM\log(e+M)}$. Set $L(a) = \log(e^e + a)$. Then,
    \[
      \Effd_n^k(\cX;\rho)
      \lesssim
      L(a)^d (\log L(a))^{-d}
    \spaced{and}
      \Gamma_n^k(\cX;\rho)
      \lesssim
      L(a)^{d+1}
      \roundb{\log L(a)}^{-d}.
    \]
  \end{enumerate}
\end{corollary}

\begin{proof}
For $b>0$, set $\Phi_\star(b):=\inf_{M\in\Np}\curlyb{M^d+b\epsilon_M^2}$. By \cref{thm:abstract-approximation-principle}, $\Effd_n^k(\cX;\rho)\lesssim\Phi_\star(a)$, while boundedness gives $\Effd_n^k(\cX;\rho)\leq\kappa_0a$. In the three cases, choose respectively
\[
  M=\ceil[\Big]{b^{1/(d+2s)}}, \qquad M=\ceil[\Big]{C\roundb{\log(e+b)}^\sigma}, \qquad M=\ceil[\Big]{C\frac{L(b)}{\log L(b)}},
\]
where $C>0$ is sufficiently large. Writing $p:=d/(d+2s)$ and $q:=d\sigma$, substitution gives
\[
  \Phi_\star(b)\lesssim1+b^p, \qquad \Phi_\star(b)\lesssim\roundb{\log(e+b)}^q, \qquad \Phi_\star(b)\lesssim L(b)^d\roundb{\log L(b)}^{-d},
\]
respectively. Together with $\Effd_n^k(\cX;\rho)\leq\kappa_0a$, these give the stated effective-dimension bounds for all $a>0$. Applying \cref{lem:information-gain-from-effective-dimension}, integrating the three profiles against $\dif b/b$ over $b\geq1$, and using the bound $\kappa_0b$ over $0<b<1$ gives the information-gain bounds.
\end{proof}

\begin{remark}[Choice of domain]
The concrete results below are stated on either $\QQ:=[0,1]^d$ or $\QQpm:=[-1,1]^d$. These choices keep the constructions explicit and avoid carrying domain geometry through the paper. Maximum effective dimension and information gain are monotone under enlargement of the design set. Consequently, whenever the corresponding kernel assumptions are preserved under fixed translations and rescalings, the same rates extend to any bounded set with nonempty interior, since such a set can be sandwiched between translated and rescaled cubes.
\end{remark}

\begin{table}[tb]
  \centering
  \small
  \caption{Sharp effective-dimension and information-gain rates for stationary kernels on~$\QQ=[0,1]^d$. Here $s$ denotes the spectral density, $a:=n/\rho$, $L(a):=\log(e^e+a)$, $\tau>0$, and $1\leq\rho\leq n$; all rates are in the sense of $\asymp$.}
  \label{tab:stationary}
  \begin{tabular}{@{}lcc@{}}
    \toprule
    Spectral regime & $\Effd_n^k(\QQ;\rho)$ & $\Gamma_n^k(\QQ;\rho)$ \\ \midrule
    \shortstack[l]{Algebraic $s(\xi)\asymp\langle\xi\rangle^{-r}$, $r>d$} & $a^{d/r}$ & $a^{d/r}$ \\ \addlinespace
    \shortstack[l]{Stretched exp. $s(\xi)\asymp\exp\curlyb{-\tau\norm{\xi}^{\gamma}}$, $0<\gamma\leq1$} \qquad\qquad& $\roundb{\log(e+a)}^{d/\gamma}$ & $\roundb{\log(e+a)}^{d/\gamma+1}$ \\ \addlinespace
    \shortstack[l]{Super-exp. $s(\xi)\asymp\exp\curlyb{-\tau\norm{\xi}^{\gamma}}$, $\gamma>1$} & $L(a)^d\roundb{\log L(a)}^{-d}$ & $L(a)^{d+1}\roundb{\log L(a)}^{-d}$ \\ \bottomrule
  \end{tabular}
\end{table}

The upper bounds are sharp in each residual regime. A stationary kernel on $\Rd$ depends only on $x-y$; throughout, we write $k(x-y)$ for $k(x,y)$. By Bochner's theorem, $k$ is the Fourier transform of a finite positive Borel measure. For the classes in \cref{tab:stationary}, this measure has a density $s$, called the spectral density, and
\[
  k(h)=\int_{\Rd}e^{2\pi i\xi\cdot h}s(\xi)\dif\xi.
\]
The two-sided bounds on $s$ combine the assumptions used for the upper and lower bounds. The algebraic and stretched-exponential spectral regimes correspond to the residual profiles with the same names, while the super-exponential regime corresponds to the factorial residual profile.

\begin{remark}[Canonical kernels and the intermediate regime]
The Mat\'ern-$\nu$ kernel belongs to the algebraic regime: its spectral density satisfies $s_\nu(\xi)\asymp\langle\xi\rangle^{-(2\nu+d)}$, so the algebraic row of \cref{tab:stationary} applies with $r=2\nu+d$. The squared exponential kernel belongs to the super-exponential regime: its spectral density satisfies $s(\xi)\asymp\exp\curlyb{-\tau\norm{\xi}^2}$ for some $\tau>0$, so the super-exponential row applies with $\gamma=2$.

The stretched-exponential class gives an intermediate scale between these two canonical examples. For each fixed $0<\gamma\leq1$, the rates in \cref{tab:stationary} are polylogarithmic, growing more slowly than every algebraic rate but more rapidly than the squared exponential rate. The boundary $\gamma=1$ marks a sharp threshold: at $\gamma=1$ there is no $\roundb{\log L(a)}^{-d}$ factor, whereas every fixed $\gamma>1$ falls in the super-exponential regime and gains this factor in both quantities. The matching lower bounds show that this transition is not an artefact of the upper-bound constructions.
\end{remark}

The next section derives the algebraic and stretched-exponential residual profiles from regularity of the native space, while \cref{sec:factorial-residual} derives the factorial profile from entire extensions. \cref{sec:main-lower-bounds} is devoted to the matching lower bounds.

\section{Residual control from native-space regularity}

Let $\cH_k$ be the native space of $k$, also called its reproducing kernel Hilbert space (RKHS), with norm $\norm{\cdot}_{\cH_k}$. For $x\in\cX$, write $k_x:=k(x,\cdot)$. Then $k_x\in\cH_k$ and
\[
  f(x)=\langle f,k_x\rangle_{\cH_k}, \qquad f\in\cH_k,\ x\in\cX.
\]
The construction uses sampling inequalities for the native space, which control functions that vanish on a finite sampling set. The following proposition converts such an inequality into the low-rank-plus-residual decomposition required by \cref{thm:abstract-approximation-principle}.

\begin{proposition}[name={},restate=PropSamplingMajorants]
\label{prop:sampling-majorants}
  Let \(\cX\subset\Rd\), and let \(k\) be a kernel on \(\cX\) with
  RKHS \(\cH_k\). Suppose that, for every \(M\in\Np\), there exist
  a finite set \(Z_M\subset\cX\) and a number \(\epsilon_M>0\) such that $\card(Z_M)\lesssim M^d$ and
  \[
    \norm{f}_{L^\infty(\cX)}
    \leq
    \epsilon_M\norm{f}_{\cH_k}
    \qquad
    \text{for all \(f\in\cH_k\) with \(f|_{Z_M}=0\)}.
  \]
  For each \(M\in\Np\), set $\cV_M := \spn\curlyb{k(z,\cdot):z\in Z_M} \subset\cH_k$, and let \(\Pi_M\) denote the orthogonal projection of \(\cH_k\)
  onto \(\cV_M\). Define
  \[
    K_M(x,y) := \left\langle \Pi_M k_y, \Pi_M k_x\right\rangle_{\cH_k}
    \spaced{and}
    R_M(x,y) := \left\langle (I-\Pi_M)k_y,(I-\Pi_M)k_x \right\rangle_{\cH_k}.
  \]
  Then, for each $M \in \Np$, \(K_M\) and \(R_M\) are positive semidefinite kernels satisfying
  \[
    k = K_M + R_M, \qquad \sup_{x \in \cX} R_M(x,x) \leq \epsilon_M^2.
  \]
  Moreover, for every $n \in \Np$ and $X_n \in \cX^n$,
  \[
    \rank G_{K_M}(X_n) \leq \dim\cV_M \leq \card(Z_M) \lesssim M^d.
  \]
\end{proposition}

\begin{proof}
Positive semidefiniteness follows directly from the Gram representations, while orthogonality of the ranges of $\Pi_M$ and $I-\Pi_M$ gives $k=K_M+R_M$. Since $f|_{Z_M}=0$ is equivalent to $f\perp\cV_M$, for every $x\in\cX$,
\[
  R_M(x,x)^{1/2}=\norm{(I-\Pi_M)k_x}_{\cH_k}=\sup\curlyb[\Big]{\abs{f(x)}:\norm{f}_{\cH_k}\leq1,\ f|_{Z_M}=0}.
\]
The assumed sampling inequality gives $R_M(x,x)\leq\epsilon_M^2$. Finally, $G_{K_M}(X_n)$ is the Gram matrix of the vectors $\Pi_Mk_{x_1},\dots,\Pi_Mk_{x_n}\in\cV_M$, and hence $\rank G_{K_M}(X_n) \lesssim M^d$.
\end{proof}

We apply \cref{prop:sampling-majorants} on the hypercube $\QQ:=[0,1]^d$ using the grids
\[
  Z_M:=M^{-1}\curlyb{0,\ldots,M}^d\subset\QQ,\qquad M\in\Np.
\]
Throughout this section, $\epsilon_M$ denotes the bound in the sampling inequality of \cref{prop:sampling-majorants} for this choice of $Z_M$.

\subsection{Algebraic residuals from finite smoothness}

Algebraic residual bounds follow from continuous embeddings of the native space into Bessel-potential or H\"older spaces. For normed function spaces $\cF$ and $\cF'$, write $\cF\hookrightarrow\cF'$ if $\cF$ is continuously embedded in $\cF'$; equivalently, there exists a constant $C>0$ such that
\[
  \norm{f}_{\cF'}\leq C\norm{f}_{\cF}
  \qquad
  \text{for all }f\in\cF.
\]
For $t>0$, let $H^t(\QQ)$ denote the restriction to $\QQ$ of the Bessel potential space $H^t(\Rd)$, equipped with its usual quotient norm. For $r\in\N$ and $\theta\in(0,1]$, let $C^{r,\theta}(\QQ)$ denote the usual H\"older space, and write $C^{r+\theta}(\QQ):=C^{r,\theta}(\QQ)$ when $0<\theta<1$. Full definitions and our conventions on open sets are given in \cref{app:functions}. With these conventions, the residual bounds are as follows.

\begin{proposition}[name={}, restate=PropFiniteSmoothnessResiduals]
\label{prop:finite-smoothness-residuals}
  Let \(k\) be a kernel on \(\QQ\) with RKHS \(\cH_k\). For
  \(M\in\Np\), let $Z_M := M^{-1}\curlyb{0,\ldots,M}^d \subset\QQ$.
  Then $\card(Z_M)\lesssim M^d$, and the following hold.

  \begin{enumerate}
    \item If \(t>d/2\) and $\cH_k\hookrightarrow H^t(\QQ)$, then $\epsilon_M \lesssim M^{-(t-d/2)}$.
    \item If \(s>0\) and $\cH_k\hookrightarrow C^s(\QQ)$, then $\epsilon_M \lesssim M^{-s}$.
    \item If $r \in \N$ and $\cH_k \hookrightarrow C^{r,1}(\QQ)$, then $\epsilon_M \lesssim M^{-(r+1)}$.
  \end{enumerate}
  Each of these embeddings also implies that $\kappa_0:=\sup_{x\in\QQ}k(x,x)<\infty$.
\end{proposition}

\begin{proof}
The grid $Z_M$ has cardinality $\card(Z_M)=(M+1)^d\lesssim M^d$ and fill distance at most $\sqrt d/(2M)$. For $f\in\cH_k$ with $f|_{Z_M}=0$, the sampling inequalities of \citet[Theorem~4.1 and Remark~0.3]{krieg2024random}, combined with the respective embeddings, give
\[
  \norm{f}_{L^\infty(\QQ)}\lesssim M^{-(t-d/2)}\norm{f}_{\cH_k},
  \ \ \norm{f}_{L^\infty(\QQ)}\lesssim M^{-s}\norm{f}_{\cH_k},
  \ \ \norm{f}_{L^\infty(\QQ)}\lesssim M^{-(r+1)}\norm{f}_{\cH_k},
\]
respectively in the three cases. This verifies the hypotheses of \cref{prop:sampling-majorants} with the stated choices of $\epsilon_M$. Finally, each assumption gives $\cH_k\hookrightarrow L^\infty(\QQ)$, whence the identity $k(x,x)^{1/2}=\sup_{\norm{f}_{\cH_k}\leq1}\abs{f(x)}$ gives $\kappa_0<\infty$.
\end{proof}

For stationary kernels, polynomial decay of the spectral density gives the Bessel-potential embedding in \cref{prop:finite-smoothness-residuals}.

\begin{proposition}[name={},restate=PropStatSobolevCond]\label{prop:stationary-sobolev-condition}
  Let \(\tilde k\) be a stationary kernel on \(\Rd\) with spectral density \(s\), and let $k := \tilde k|_{\QQ\times\QQ}$. Suppose that for some \(t > d/2\),
  \[
    s(\xi) \leq C \langle \xi \rangle^{-2t}
    \qquad
    \text{for all } \xi \in \Rd.
  \]
  Then $\cH_k\hookrightarrow H^t(\QQ)$, and therefore $\epsilon_M \lesssim M^{-(t-d/2)}$ on the uniform grid on $\QQ$.
\end{proposition}

\begin{proof}
  By the standard Fourier characterisation of stationary RKHSs, the decay assumption on \(s\) implies \(\cH_{\tilde k}\hookrightarrow H^t(\Rd)\); see \citet[Theorem~10.12]{wendland2004scattered}. The native space of the restriction of a kernel is the restriction of the native space, with the corresponding quotient norm (\cref{lem:rkhs-restriction}). Combining these facts gives \(\cH_k\hookrightarrow H^t(\QQ)\).
\end{proof}

The proposition yields the rates for the Mat\'ern kernel.

\begin{corollary}[name={},restate=CorMaternUB]\label{cor:matern-sobolev-upper}
  Let \(k_\nu\) be the Mat\'ern-\(\nu\) kernel on \(\Rd\), \(\nu>0\), so that its spectral density satisfies
  \[
    s_\nu(\xi) \asymp \langle \xi \rangle^{-(2\nu+d)}.
  \]
  Then, for every \(n\geq1\) and \(\rho>0\),
  \[
    \Effd_n^{k_\nu}(\QQ;\rho)
    \leq
    C (n/\rho)^{d/(2\nu+d)} \spaced{and}
    \Gamma_n^{k_\nu}(\QQ;\rho)
    \leq
    C (n/\rho)^{d/(2\nu+d)}.
  \]
\end{corollary}

\begin{proof}
    The corollary follows by applying the proposition with $t = \nu + d/2$, then combining with \cref{prop:finite-smoothness-residuals}  and the algebraic residual part of \cref{cor:canonical-residual-profiles}.
\end{proof}

The H\"older branch of \cref{prop:finite-smoothness-residuals} can be verified directly from a feature map.

\begin{proposition}[
  name={},
  restate=PropFeatureHolder
]
\label{prop:feature-map-holder-native}
  Fix $r \in \N$ and $\theta \in [0,1)$ with $r + \theta > 0$. Let \(U\subset\Rd\) be open with \(\QQ\subset U\), let
  \(\tilde k\) be a kernel on \(U\), and suppose that
  \[
    \tilde k(x,y)
    =
    \langle \Phi(y),\Phi(x)\rangle_{\cF},
    \qquad
    x,y\in U,
  \]
  for a Hilbert space \(\cF\) and a feature map \(\Phi\colon U\to\cF\). Set $k:=\tilde k|_{\QQ\times\QQ}$.

  Suppose that \(\Phi\) is \(r\)-times continuously differentiable
  on a neighbourhood of \(\QQ\), where derivatives are taken in the
  norm of \(\cF\). If \(\theta>0\), suppose in addition that
  \[
    \norm{
      \partial^\alpha\Phi(x)
      -
      \partial^\alpha\Phi(y)
    }_{\cF}
    \lesssim
    \norm{x-y}^{\theta}
    \qquad
    \text{for all \(x,y\in\QQ\) and \(|\alpha|=r\)}.
  \]
  Then $\cH_k\hookrightarrow C^{r+\theta}(\QQ)$, and therefore $\epsilon_M \lesssim M^{-(r+\theta)}$ on the uniform grid on $\QQ$.
\end{proposition}

\begin{proof}
By \cref{lem:feature-space-representation}, for every
\(f\in\cH_k\) there exists \(v_f\in\cF\) such that
\[
  f(x)=\langle v_f,\Phi(x)\rangle_{\cF},
  \qquad
  x\in\QQ,
  \qquad
  \norm{v_f}_{\cF}=\norm{f}_{\cH_k}.
\]
  Let \(W\subset U\) be an open neighbourhood of \(\QQ\) on which \(\Phi\) is \(r\)-times continuously differentiable, and define \(\bar f(x):=\langle v_f,\Phi(x)\rangle_{\cF}\) for \(x\in W\). Then \(\bar f|_{\QQ}=f\). Since \(v\mapsto\langle v_f,v\rangle_{\cF}\) is bounded and linear, repeated application of the chain rule \citep[p.~507 and Lemma~A.5.15(ii)]{steinwart2008support} gives
  \[
    \partial^\alpha\bar f(x)=\langle v_f,\partial^\alpha\Phi(x)\rangle_{\cF},\qquad x\in W,\quad |\alpha|\leq r.
  \]
  Continuity of \(\partial^\alpha\Phi\) on a neighbourhood of the compact set \(\QQ\), followed by Cauchy--Schwarz, yields
  \[
    \norm{\partial^\alpha f}_{L^\infty(\QQ)}\leq\norm{f}_{\cH_k}\sup_{x\in\QQ}\norm{\partial^\alpha\Phi(x)}_{\cF}\lesssim\norm{f}_{\cH_k},\qquad |\alpha|\leq r.
  \]
  If \(\theta>0\) and \(|\alpha|=r\), then
  \[
    \abs{\partial^\alpha f(x)-\partial^\alpha f(y)}\leq\norm{f}_{\cH_k}\norm{\partial^\alpha\Phi(x)-\partial^\alpha\Phi(y)}_{\cF}\lesssim\norm{f}_{\cH_k}\norm{x-y}^{\theta}.
  \]
  Hence \(\norm{f}_{C^{r+\theta}(\QQ)}\lesssim\norm{f}_{\cH_k}\), proving \(\cH_k\hookrightarrow C^{r+\theta}(\QQ)\). The residual bound follows from \cref{prop:finite-smoothness-residuals}.
\end{proof}

  A criterion stated directly in terms of smoothness of the scalar kernel is simpler to verify, but can be less sharp than the feature-map criterion: $C^s$-regularity of the kernel yields only $C^{s/2}$-regularity of the native space.

\begin{proposition}[
  name={},
  restate=CorSmoothKernelResiduals
]
\label{cor:smooth-kernel-residuals}
  Let \(s>0\), let \(U\subset\Rd\) be open with \(\QQ\subset U\), and
  let \(\tilde k\) be a positive semidefinite kernel on \(U\) such that $\tilde k\in C^s(U\times U)$.
  Set $k:=\tilde k|_{\QQ\times\QQ}$.
  Then $\cH_k\hookrightarrow C^{s/2}(\QQ)$,
  and the residual satisfies the bound $\epsilon_M\lesssim M^{-s/2}$. Moreover, $\kappa_0 := \sup_{x\in\QQ}k(x,x) < \infty$.
\end{proposition}

\begin{proof}[Proof sketch]
  The derivative reproducing formula (\cref{lem:derivative-reproducing}) converts smoothness of \(\tilde k\) into uniform derivative bounds for functions in \(\cH_k\), while increments of the mixed kernel derivatives give the required H\"older control at the highest derivative order. This yields \(\cH_k\hookrightarrow C^{s/2}(\QQ)\), and the residual bound follows from \cref{prop:finite-smoothness-residuals}. See \cref{app:holder-continuity} for the full proof.
\end{proof}

\begin{remark}
  The feature-map criterion gives an alternative derivation of the
  Mat\'ern residual when the smoothness parameter is noninteger. Let
  \(\nu>0\) with \(\nu\notin\Np\), write
  \[
    \nu=r+\theta,
    \qquad
    r:=\lfloor\nu\rfloor,
    \qquad
    0<\theta<1,
  \]
  and let \(\tilde k_\nu\) be a Mat\'ern-\(\nu\) kernel on \(\Rd\), whose spectral density satisfies $s_\nu(\xi) \asymp \langle\xi\rangle^{-(2\nu+d)}$. Consider the weighted $L^2$ space $\cF_\nu := L^2\roundb{\Rd,s_\nu(\xi)\dif\xi}$ and the spectral feature map
  $\Phi_\nu(x)(\xi) := e^{2\pi i\xi\cdot x}$. Then
  $\tilde k_\nu(x,y) = \langle\Phi_\nu(y),\Phi_\nu(x)\rangle_{\cF_\nu}$.
  For every multi-index \(\alpha\) with \(|\alpha|\leq r\), $\partial^\alpha\Phi_\nu(x)(\xi)
    =
    (2\pi i\xi)^\alpha e^{2\pi i\xi\cdot x}$,
  and
  \[
    \sup_{x\in\QQ}
    \norm{\partial^\alpha\Phi_\nu(x)}_{\cF_\nu}^2
    \lesssim
    \int_{\Rd}
    \norm{\xi}^{2|\alpha|}
    \langle\xi\rangle^{-(2\nu+d)}
    \dif\xi
    <
    \infty,
  \]
  since \(|\alpha|\leq r<\nu\). Now fix \(|\alpha|=r\), let \(x,y\in\QQ\), and put
  \(h:=\norm{x-y}\). For \(0<h\leq1\),
  \begin{align*}
    \norm{
      \partial^\alpha\Phi_\nu(x)
      -
      \partial^\alpha\Phi_\nu(y)
    }_{\cF_\nu}^2
    &\lesssim
    \int_{\Rd}
    \norm{\xi}^{2r}
    \roundb{1\wedge h^2\norm{\xi}^2}
    \langle\xi\rangle^{-(2\nu+d)}
    \dif\xi
    \\
    &
    \lesssim
    h^2
    +
    h^2\int_1^{h^{-1}}u^{1-2\theta}\dif u
    +
    \int_{h^{-1}}^\infty u^{-2\theta-1}\dif u
    \lesssim
    h^{2\theta}.
  \end{align*}
  For \(h>1\), the same conclusion follows from the uniform derivative
  bound. Thus
  \[
    \norm{
      \partial^\alpha\Phi_\nu(x)
      -
      \partial^\alpha\Phi_\nu(y)
    }_{\cF_\nu}
    \lesssim
    \norm{x-y}^{\theta}.
  \]
  Applying \cref{prop:feature-map-holder-native} gives $\cH_{k_\nu}\hookrightarrow C^\nu(\QQ)$,
  where \(k_\nu:=\tilde k_\nu|_{\QQ\times\QQ}\).

  At integer values of $\nu$, the above integral is borderline and produces a logarithmic modulus rather than a Lipschitz estimate. The Bessel potential route is therefore cleaner.
\end{remark}

\subsection{Stretched-exponential residuals from ultradifferentiability}

The algebraic residual bounds above use regularity of a fixed order. Bounds on derivatives of every order with controlled factorial growth yield stretched-exponential residuals.

\begin{proposition}[
  name={},
  restate=PropGevreyResiduals
]
\label{prop:gevrey-native-space-residual}
  Let \(\sigma\geq1\), let \(k\) be a kernel on \(\QQ\), and let
  \(\cH_k\) be its RKHS. Suppose that there exists \(A\geq1\) such that
  \[
    \norm{\partial^\alpha f}_{L^\infty(\QQ)}
    \leq
    A^{|\alpha|+1}(\alpha!)^\sigma\norm{f}_{\cH_k}
    \qquad
    \text{for all \(\alpha\in\Nd\) and \(f\in\cH_k\)}.
  \]
  Then, for the grids \(Z_M\) of
  \cref{prop:finite-smoothness-residuals}, there exist constants
  \(C,c>0\) such that
  \[
    \norm{f}_{L^\infty(\QQ)}
    \leq
    C\exp\curlyb{-cM^{1/\sigma}}\norm{f}_{\cH_k}
    \qquad
    \text{for all \(f\in\cH_k\) with \(f|_{Z_M}=0\)}.
  \]
  Thus, in \cref{prop:sampling-majorants}, one may take
  \(\epsilon_M\lesssim\exp\curlyb{-cM^{1/\sigma}}\), and $\sup_{x\in\QQ}k(x,x)<\infty$.
\end{proposition}

\begin{proof}
For $m\in\Np$, since $\QQ$ has unit volume, $\alpha!\leq|\alpha|!\leq m!$ whenever $|\alpha|\leq m$, and $\card\curlyb{\alpha\in\Nd:|\alpha|\leq m}=\binom{m+d}{d}$, the assumed derivative bounds give
\[
  \curlyb[\bigg]{\sum_{|\alpha|\leq m}\norm{\partial^\alpha f}_{L^2(\QQ)}^2}^{1/2}\leq A^{m+1}\binom{m+d}{d}^{1/2}(m!)^\sigma\norm{f}_{\cH_k}\leq B^m m^{\sigma m}\norm{f}_{\cH_k}
\]
for some $B\geq1$ depending only on $A$ and $d$. Thus the order-$m$ embedding constants satisfy the hypothesis of \citet[Theorem~4.5]{rieger2010sampling} with $s=\sigma$. Writing $h_M:=\sup_{x\in\QQ}\min_{z\in Z_M}\norm{x-z}$, we have $h_M\leq\sqrt d/(2M)$. Applying that theorem with $p=2$ and $q=\infty$, and using $f|_{Z_M}=0$, gives constants $c>0$ and $M_0\in\Np$ such that $\norm{f}_{L^\infty(\QQ)}\leq\exp\curlyb{-cM^{1/\sigma}}\norm{f}_{\cH_k}$ for every $M\geq M_0$. The case $\alpha=0$ gives $\norm{f}_{L^\infty(\QQ)}\leq A\norm{f}_{\cH_k}$, which covers the remaining values of $M$ after increasing the multiplicative constant and gives $\sup_{x\in\QQ}k(x,x)\leq A^2$.
\end{proof}

The derivative condition above is naturally expressed in terms of Gevrey classes. Let $V\subset\R^q$ be open and let $\sigma\geq1$. We say that $g\in C^\infty(V)$ belongs to the Gevrey class $G^\sigma(V)$ if, for every compact $K\subset V$, there exists $C_K\geq1$ such that
\[
  \abs{\partial^\alpha g(x)}\leq C_K^{|\alpha|+1}(\alpha!)^\sigma,\qquad x\in K,\ \alpha\in\N^q.
\]
The class $G^1(V)$ is the real-analytic class. For $\sigma>1$, the condition permits faster derivative growth, and is therefore weaker than analyticity while still controlling derivatives of every order. A positive semidefinite Gevrey extension yields the derivative bounds in \cref{prop:gevrey-native-space-residual}.

\begin{corollary}[
  name={},
  restate=CorGevreyKernel
]
\label{cor:gevrey-kernel-residual}
  Let \(U\subset\Rd\) be open with \(\QQ\subset U\), let \(\sigma\geq1\), and let \(\tk\) be a positive semidefinite kernel on \(U\). Suppose that \(\tk\in G^\sigma(U\times U)\), and set
  \(k:=\tk|_{\QQ\times\QQ}\). Then the hypothesis of \cref{prop:gevrey-native-space-residual} holds. Consequently, $\epsilon_M\lesssim\exp\curlyb{-cM^{1/\sigma}}$ and $\kappa_0:=\sup_{x\in\QQ}k(x,x)<\infty$.
\end{corollary}
\begin{proof}
Since $\QQ\times\QQ$ is compact in $U\times U$, there exists $B\geq1$ such that $\abs{\partial_x^\alpha\partial_y^\alpha\tk(x,x)}\leq B^{2|\alpha|+1}(\alpha!)^{2\sigma}$ for all $x\in\QQ$ and $\alpha\in\Nd$. Fix $f\in\cH_k$ and use \cref{lem:rkhs-restriction} to choose $F\in\cH_{\tk}$ such that $F|_{\QQ}=f$ and $\norm{F}_{\cH_{\tk}}=\norm{f}_{\cH_k}$. By \cref{lem:derivative-reproducing},
\[
  \abs{\partial^\alpha f(x)} \leq \norm{F}_{\cH_{\tk}}\norm{\partial_x^\alpha\tk(x,\cdot)}_{\cH_{\tk}} = \norm{f}_{\cH_k}\roundb{\partial_x^\alpha\partial_y^\alpha\tk(x,x)}^{1/2} \leq B^{|\alpha|+1}(\alpha!)^\sigma\norm{f}_{\cH_k}.
\]
Taking the supremum over $x\in\QQ$ verifies the hypothesis of \cref{prop:gevrey-native-space-residual}; the residual and bounded-diagonal conclusions follow.
\end{proof}

For stationary kernels, an exponential moment of the spectral measure yields the Gevrey regularity required above.

\begin{proposition}[
  name={},
  restate=PropExpMomentGev
]
\label{prop:exp-moment-gevrey}
  Let \(0<\gamma\leq1\), let \(\mu\) be a finite symmetric positive
  Borel measure on \(\Rd\), and suppose that, for some \(\tau_0>0\),
  \[
    \int_{\Rd}
    \exp\curlyb{\tau_0\norm{\xi}^{\gamma}}
    \mu(\dif\xi)
    <
    \infty.
  \]
Define $K(h):=\int_{\Rd}e^{2\pi i\xi\cdot h}\mu(\dif\xi)$, $\tk(x,y):=K(x-y)$, and $k:=\tk|_{\QQ\times\QQ}$. Then
  \(\tk\in G^{1/\gamma}(\Rd\times\Rd)\), and the sampling construction
  has residual
  \(\epsilon_M\lesssim\exp\curlyb{-cM^\gamma}\).

  In particular, if \(\mu\) has density
  \(s(\xi)=\exp\curlyb{-\tau\norm{\xi}^{\gamma}}\), where \(\tau>0\),
  then, for every \(n\in\Np\) and \(\rho>0\),
  \[
    \Effd_n^k(\QQ;\rho)
    \lesssim
    \roundb{\log(e+n/\rho)}^{d/\gamma}
    \spaced{and}
    \Gamma_n^k(\QQ;\rho)
    \lesssim
    \roundb{\log(e+n/\rho)}^{d/\gamma+1}.
  \]
\end{proposition}

\begin{proof}[Proof sketch]
The exponential moment gives $\int_{\Rd}\norm{\xi}^m\mu(\dif\xi)\leq B^{m+1}(m!)^{1/\gamma}$ for every $m\in\N$. Differentiating the Fourier representation under the integral then gives $\tk\in G^{1/\gamma}(\Rd\times\Rd)$, and the residual and complexity bounds follow from \cref{cor:gevrey-kernel-residual,cor:canonical-residual-profiles}. See \cref{app:stretched-exp-moments} for the details.
\end{proof}

The super-exponential regime $\gamma>1$ is treated through entire extensions in the next section.

\section{Factorial residuals from entire extensions}
\label{sec:factorial-residual}

The factorial regime is obtained by directly approximating an entire extension of the kernel. We work on $\QQpm:=[-1,1]^d$. For \(\varrho>1\), let \(\EE_\varrho\subset\bC\) denote the closed region bounded by the Bernstein ellipse,
\[
  \EE_\varrho
  :=
  \curlyb[\Big]{
    z\in\bC:
    \abs{z-1}+\abs{z+1}
    \leq
    \varrho+\varrho^{-1}
  }.
\]
Its boundary is the image of the circle \(\abs{u}=\varrho\) under the
map \(u\mapsto\frac12(u+u^{-1})\). We write \(\EE_\varrho^q\) for the
\(q\)-fold product of \(\EE_\varrho\).

The following proposition converts control of the extension on products of Bernstein ellipses into the low-rank majorant required by \cref{thm:abstract-approximation-principle}.

\begin{proposition}[
  name={},
  restate=PropEntireMajorants
]
\label{prop:entire-finite-rank-majorants}
  Let \(k\) be a positive semidefinite kernel on \(\QQpm\), and let
  \(\tk\colon\bC^d\times\bC^d\to\bC\) be an entire extension of \(k\).
  For \(\varrho>1\), set
  \[
    B_{\tk}(\varrho)
    :=
    \sup_{z,w\in\EE_\varrho^d}\abs{\tk(z,w)}.
  \]
  Then, for every \(M\in\Np\) and \(\varrho\geq2\), there exist
  positive semidefinite kernels \(K_{M,\varrho}\) and
  \(R_{M,\varrho}\) such that, for every \(n\in\Np\) and
  \(X_n\in(\QQpm)^n\),
  \[
    G_k(X_n)
    \preceq
    G_{K_{M,\varrho}}(X_n)+G_{R_{M,\varrho}}(X_n),
    \qquad
    \rank G_{K_{M,\varrho}}(X_n)\leq M^d,
  \]
  and
  \[
    \sup_{x\in\QQpm}R_{M,\varrho}(x,x)
    \lesssim
    B_{\tk}(\varrho)\varrho^{-M}.
  \]
\end{proposition}

\begin{proof}[Proof sketch]
  Expand the entire function \(\tk(z,w)\) in tensor-product Chebyshev polynomials on \(\EE_\varrho^d\times\EE_\varrho^d\). Holomorphy on this polyellipse gives coefficients \(a_{\alpha,\beta}\) satisfying \(\abs{a_{\alpha,\beta}}\lesssim B_{\tk}(\varrho)\varrho^{-|\alpha|-|\beta|}\). Replacing each mixed term by a diagonal positive semidefinite majorant produces a kernel \(H_\varrho\) whose Gram matrices dominate those of \(k\). Splitting the expansion of \(H_\varrho\) at coordinate degree \(M\) gives a sum of \(M^d\) rank-one kernels and a positive semidefinite remainder whose diagonal is bounded by the geometric tail \(B_{\tk}(\varrho)\varrho^{-M}\). See \cref{app:entire-moments} for the full proof.
\end{proof}

Thus, for any sequence $\varrho_M\geq2$, the hypotheses of \cref{thm:abstract-approximation-principle} hold with
\[
  \epsilon_M^2\lesssim B_{\tk}(\varrho_M)\varrho_M^{-M}.
\]
The factorial profile follows by balancing the approximation factor $\varrho_M^{-M}$ against a finite-order growth bound for $B_{\tk}(\varrho_M)$.

 We say that a kernel \(k\) on $\QQpm$ is entire of finite order if it admits an entire extension \(\tk\colon\bC^d\times\bC^d\to\bC\) for which there exist \(C_0,c_0,p>0\) such that
 \[
    \abs{\tk(z,w)} \leq C_0\exp\curlyb{c_0\max\roundb{\norm{z}_\infty^p, \norm{w}_\infty^p}} \qquad \text{for all } z,w\in\bC^d.
  \]
 If the bound holds for a given \(p\), we say that \(k\) has order at most \(p\).

\begin{proposition}[
  name={},
  restate=CorFiniteOrderResiduals
]
\label{cor:finite-order-entire-residual}
  Let \(k\) be a positive semidefinite kernel on \(\QQpm\) which is
  entire of finite order. Then the hypotheses of
  \cref{thm:abstract-approximation-principle} hold with
  \[
    \epsilon_M
    \lesssim
    \exp\curlyb{-cM\log(e+M)}
  \]
  for some \(c>0\). Moreover,
  \(\kappa_0:=\sup_{x\in\QQpm}k(x,x)<\infty\).
\end{proposition}

\begin{proof}
  Let \(\tk\) be an entire extension for which
  \[
    \abs{\tk(z,w)}
    \leq
    C_0
    \exp\curlyb{
      c_0\max\roundb{
        \norm{z}_\infty^p,
        \norm{w}_\infty^p
      }
    },
    \qquad
    z,w\in\bC^d.
  \]
  If \(z\in\EE_\varrho^d\), then $\norm{z}_\infty \leq \frac12\roundb{\varrho+\varrho^{-1}} \leq \varrho$. Consequently,
  \[
    B_{\tk}(\varrho) \leq C_0\exp\curlyb{c_0\varrho^p},
    \qquad \varrho>1.
  \]

  For \(M\in\Np\), set $\varrho_M := \max\curlyb{2,M^{1/p}}$. By \cref{prop:entire-finite-rank-majorants}, the hypotheses of
  \cref{thm:abstract-approximation-principle} hold with
  \[
    \epsilon_M^2
    \leq
    C_dC_0
    \exp\curlyb{
      c_0\varrho_M^p-M\log\varrho_M
    }.
  \]
  For all sufficiently large \(M\), \(\varrho_M=M^{1/p}\), and hence
  \[
    c_0\varrho_M^p-M\log\varrho_M
    =
    c_0M-\frac{M}{p}\log M
    \leq
    -\frac{M}{2p}\log M.
  \]
  Taking square roots, changing constants to cover the finitely many
  remaining values of \(M\), and replacing \(\log M\) by
  \(\log(e+M)\), we obtain $\epsilon_M \lesssim \exp\curlyb{-cM\log(e+M)}$ for some \(c>0\).

  Finally, \(k\) is continuous on
  \(\QQpm\times\QQpm\), since it is the restriction of the entire
  function \(\tk\). Therefore
  $\kappa_0 = \max_{x\in\QQpm}k(x,x) < \infty$.
\end{proof}

For stationary kernels, a super-exponential moment of the spectral measure gives the required finite-order extension.

\begin{proposition}[
  name={},
  restate=PropStatFiniteOrder
]
\label{prop:stationary-entire-finite-order}
  Let \(k\) be a stationary kernel on \(\Rd\) with spectral measure
  \(\mu\). Suppose that, for some \(\tau_0>0\) and \(\gamma>1\),
  \[
    \int_{\Rd}
    \exp\curlyb{\tau_0\norm{\xi}^{\gamma}}
    \mu(\dif\xi)
    <
    \infty.
  \]
  Then the restriction of \(k\) to
  \(\QQpm\times\QQpm\) is entire of finite order at most
  \(p:=\gamma/(\gamma-1)\), and the construction of
  \cref{prop:entire-finite-rank-majorants} has residual $\epsilon_M
    \lesssim
    \exp\curlyb{-cM\log(e+M)}$.
\end{proposition}

\begin{proof}
  Proof given in \cref{app:entire-moments}.
\end{proof}

The model spectral densities in the super-exponential regime therefore give the following bounds.

  \begin{corollary}
Let $k$ be the stationary kernel on $\Rd$ generated by the spectral density
\[
  s(\xi)=\exp\curlyb{-\tau\norm{\xi}^{\gamma}},\qquad \gamma>1,\ \tau>0.
\]
 Then, for every \(n\in\Np\) and \(\rho>0\), writing
  \(a:=n/\rho\) and \(L(a):=\log(e^e+a)\),
  \[
    \Effd_n^k(\QQ;\rho)
    \lesssim
    L(a)^d\roundb{\log L(a)}^{-d}
    \spaced{and}
    \Gamma_n^k(\QQ;\rho)
    \lesssim
    L(a)^{d+1}\roundb{\log L(a)}^{-d}.
  \]
\end{corollary}

\begin{proof}
The moment condition holds with $\tau_0=\tau/2$, so \cref{prop:stationary-entire-finite-order} gives factorial residuals on $\QQpm$. Restricting the resulting majorants to $\QQ$ and applying \cref{cor:canonical-residual-profiles} gives the displayed bounds.
\end{proof}

\section{Lower bounds for stationary kernels}
\label{sec:main-lower-bounds}

The upper bounds follow by showing that every Gram matrix is
approximately supported on a low-dimensional subspace. For the lower
bounds, we construct designs whose Gram matrices are bounded below on
a high-dimensional subspace. The following proposition converts such
a subspace into lower bounds for effective dimension and
information gain.

\begin{proposition}
\label{prop:resolved-directions}
  Let \(A\in\bC^{n\times n}\) be positive semidefinite, let
  \(1\leq m\leq n\), and let \(U\in\bC^{n\times m}\) satisfy
  \(U^*U=I_m\). Suppose that, for some \(\rho,\theta>0\), $U^*AU \succeq \rho\theta I_m$. Then
  \[
    \tr\roundb[]{A(A+\rho I_n)^{-1}}
    \geq
    m\frac{\theta}{1+\theta}
    \spaced{and}
    \log\det\roundb{I_n+\rho^{-1}A}
    \geq
    m\log(1+\theta).
  \]
\end{proposition}

\begin{proof}
  Let $\lambda_1\geq\cdots\geq\lambda_n\geq0$ and $\mu_1\geq\cdots\geq\mu_m\geq0$ denote the eigenvalues of \(A\) and \(U^*AU\), respectively. By the Courant--Fischer formula, for \(j\in[m]\),
  \[
    \lambda_j = \max_{\substack{V\subset\bC^n\\ \dim V=j}} \min_{\substack{v\in V\\ \norm{v}=1}} v^*Av \geq \max_{\substack{V\subset\operatorname{ran}(U)\\ \dim V=j}} \min_{\substack{v\in V\\ \norm{v}=1}} v^*Av = \mu_j.
  \]
  The hypothesis gives \(\mu_j\geq\rho\theta\), and hence
  \(\lambda_j\geq\rho\theta\), for every \(j\in[m]\). Therefore
  \[
    \tr\roundb[]{A(A+\rho I_n)^{-1}} = \sum_{j=1}^n\frac{\lambda_j}{\lambda_j+\rho} \geq \sum_{j=1}^m\frac{\lambda_j}{\lambda_j+\rho} \geq m\frac{\theta}{1+\theta},
  \]
  and
  \[
    \log\det\roundb{I_n+\rho^{-1}A} = \sum_{j=1}^n\log\roundb{1+\lambda_j/\rho} \geq m\log(1+\theta). \qedhere
  \]
\end{proof}

We compare each kernel with a stationary witness kernel. The following lemma transfers pointwise domination of spectral densities to Gram-matrix domination.

\begin{lemma}[Spectral domination]
\label{lem:spectral-domination}
  Let \(k\) and \(K\) be stationary kernels on \(\Rd\) with spectral
  densities \(s\) and \(S\), respectively. Suppose that, for some
  \(c_0>0\),
  \[
    s(\xi)
    \geq
    c_0S(\xi)
    \qquad
    \text{for almost every }\xi\in\Rd.
  \]
  Then, for every \(n\in\Np\) and \(X_n\in(\Rd)^n\),
  \[
    G_k(X_n)
    \succeq
    c_0G_K(X_n).
  \]
\end{lemma}

\begin{proof}
  Fix \(X_n=(x_1,\dots,x_n)\in(\Rd)^n\). By the
  spectral representations of \(k\) and \(K\),
  \[
    v^*
    \roundb{G_k(X_n)-c_0G_K(X_n)}
    v
    =
    \int_{\Rd}
    \roundb{s(\xi)-c_0S(\xi)}
    \abs{
      \sum_{j=1}^n
      v_j e^{-2\pi i\xi\cdot x_j}
    }^2
    \dif\xi
    \geq
    0
  \]
  for every \(v\in\bC^n\). This proves the claimed matrix inequality.
\end{proof}

The algebraic and stretched-exponential bounds use Fourier directions on a uniform lattice. For bounded values of $a=n/\rho$, all three results follow by taking all design points equal; the sketches below describe the large-$a$ constructions.

\begin{theorem}[
  name={},
  restate=ThmPolyLB
]
\label{thm:polynomial-lower-bound}
  Let \(r>d\), and let \(k\) be a stationary kernel on \(\Rd\) with
  spectral density \(s\). Suppose that there exists \(c_0>0\) such
  that
  \[
    s(\xi)
    \geq
    c_0\langle\xi\rangle^{-r}
    \qquad
    \text{for almost every }\xi\in\Rd.
  \]
  Then there exists \(C>0\) such that, for every
  \(n\in\Np\) and \(1 \leq \rho \leq n\),
  writing \(a:=n/\rho\),
  \[
    \Effd_n^k(\QQ;\rho)
    \geq
    Ca^{d/r}
    \spaced{and}
    \Gamma_n^k(\QQ;\rho)
    \geq
    Ca^{d/r}.
  \]
\end{theorem}

\begin{proof}[Proof sketch]
We compare $k$ with a stationary witness kernel $K$ whose spectral density satisfies $S(\xi)\asymp\langle\xi\rangle^{-r}$ and whose spatial kernel is supported in $(-1,1)^d$. By \cref{lem:spectral-domination}, it suffices to treat $K$. On a uniform lattice with $n_0\asymp n$ points, its Gram matrix agrees with that of its $2$-periodisation. On this lattice, the Fourier mode $2m$ agrees with the direction indexed by $m$, and its coefficient is $2^{-d}S(m)$. Taking $\cM=\{0,\ldots,M-1\}^d$ with $M\asymp a^{1/r}$ gives $n_0 2^{-d}S(m)\gtrsim\rho$ for every $m\in\cM$ and $\card(\cM)\asymp a^{d/r}$. The corresponding orthonormal directions therefore satisfy the hypothesis of \cref{prop:resolved-directions} with $\theta\gtrsim1$. See \cref{app:algebraic-lb} for the full proof.
\end{proof}

The same lattice construction applies in the stretched-exponential regime, with spatial decay controlling the error introduced by periodisation.

\begin{theorem}[
  name={},
  restate=ThmGenExpLB
]
\label{thm:generalised-exp-lower-bound}
  Let \(0<\gamma\leq1\), let \(\tau>0\), and let \(k\) be a stationary
  kernel on \(\Rd\) with spectral density \(s\). Suppose that there
  exists \(c_0>0\) such that
  \[
    s(\xi)
    \geq
    c_0\exp\curlyb{-\tau\norm{\xi}^{\gamma}}
    \qquad
    \text{for almost every }\xi\in\Rd.
  \]
  Then there exists \(C>0\) such that, for every
  \(n\in\Np\) and \(1 \leq \rho \leq n\),
  writing \(a:=n/\rho\),
  \[
    \Effd_n^k(\QQ;\rho)
    \geq
    C\roundb{\log(e+a)}^{d/\gamma}
    \spaced{and}
    \Gamma_n^k(\QQ;\rho)
    \geq
    C\roundb{\log(e+a)}^{d/\gamma+1}.
  \]
\end{theorem}

\begin{proof}[Proof sketch]
We compare $k$ with the stationary witness kernel $K$ having spectral density $S(\xi)=\exp\curlyb{-\tau\norm{\xi}^{\gamma}}$. By \cref{lem:spectral-domination}, it suffices to treat $K$. On a uniform lattice with $n_0\asymp n$ points, periodise $K$ with an integer period $T\asymp a^\alpha$, where $1/(d+\gamma)<\alpha<1/d$. On this lattice, the Fourier mode $Tm$ agrees with the direction indexed by $m$ and has coefficient $T^{-d}S(m)$. The decay $\abs{K(x)}\lesssim\langle x\rangle^{-(d+\gamma)}$ bounds the difference between the original and periodised Gram matrices in operator norm by $\lesssim nT^{-(d+\gamma)}$, which is smaller than $\rho$ for sufficiently large $a$. Taking $\cM=\{0,\ldots,M-1\}^d$ with $M\asymp(\log a)^{1/\gamma}$ gives $n_0T^{-d}S(m)\gtrsim\rho a^\beta$ for every $m\in\cM$ and some $\beta>0$, while $\card(\cM)\asymp(\log a)^{d/\gamma}$. Applying \cref{prop:resolved-directions} gives a constant-order contribution per direction to effective dimension and a contribution of order $\log a$ per direction to information gain. See \cref{app:sub-exp-lb} for the full proof.
\end{proof}

The super-exponential regime uses a different construction based on leading Mercer directions.

\begin{theorem}[
  name={},
  restate=ThmSupExpLB
]
\label{thm:super-exp-model-kernels-lower}
  Let \(\gamma>1\), let \(\tau>0\), and let \(k\) be a stationary
  kernel on \(\Rd\) with spectral density \(s\). Suppose that there
  exists \(c_0>0\) such that
  \[
    s(\xi)
    \geq
    c_0\exp\curlyb{-\tau\norm{\xi}^{\gamma}}
    \qquad
    \text{for almost every }\xi\in\Rd.
  \]
  Then there exists \(C>0\) such that, for every
  \(n\in\Np\) and \(1 \leq \rho \leq n\),
  writing \(a:=n/\rho\) and \(L(a):=\log(e^e+a)\),
  \[
    \Effd_n^k(\QQ;\rho) \geq C L(a)^d(\log L(a))^{-d}
    \spaced{and} \Gamma_n^k(\QQ;\rho) \geq C\roundb{L(a)}^{d+1} \roundb{\log L(a)}^{-d}.
  \]
\end{theorem}

\begin{proof}[Proof sketch]
By spectral domination and equivalence of norms on $\Rd$, it suffices to consider the product witness kernel $K_\gamma$ with spectral density $S_\gamma(\xi)=\exp\curlyb{-\tau_0\sum_{\ell=1}^d\abs{\xi_\ell}^{\gamma}}$. One-dimensional eigenvalue asymptotics of \citet{widom1964asymptotic}, together with a product counting argument, show that its Mercer eigenvalues satisfy $-\log\lambda_m\asymp m^{1/d}\log m$. Taking $M\asymp L(a)^d\roundb{\log L(a)}^{-d}$ with a sufficiently small constant gives $\lambda_M\geq a^{-1/4}$. Let $(\lambda_j,\phi_j)$ be the Mercer eigenpairs of $K_\gamma$ and let $V\in\R^{n\times M}$ have entries $V_{ij}:=\lambda_j^{1/2}\phi_j(x_i)$ for independent uniform points $x_1,\ldots,x_n\in\QQ$. Then $\E[n^{-1}V\tran V]=\diag(\lambda_1,\ldots,\lambda_M)$ and $\sum_{j=1}^M\lambda_j\phi_j(x)^2\leq K_\gamma(x,x)$. Since $n\lambda_M^2\gtrsim a^{1/2}$ dominates $\log M$, the matrix Hoeffding inequality gives a realisation for which $V\tran V\succeq n\lambda_M I_M/2$. Setting $Q:=V(V\tran V)^{-1/2}$, positivity of the remaining Mercer terms and spectral domination give $Q\tran G_kQ\succeq cn\lambda_MI_M$. Thus the hypothesis of \cref{prop:resolved-directions} holds with $\theta\gtrsim a\lambda_M\gtrsim a^{3/4}$, which gives the stated effective-dimension bound and an additional factor of order $L(a)$ for information gain. See \cref{app:sup-exp-lb} for the full proof.
\end{proof}

Thus each upper-bound regime is attained by a corresponding stationary kernel class, giving the two-sided rates in \cref{tab:stationary}.

\section{Related work}\label{sec:related-work}

Previous works proceeded by bounding the information, with bounds on the effective dimension recovered through the estimate
\[
  \effd^k(X_n;\rho)
  \leq
  \gamma^k(X_n;\rho)
  \leq
  \effd^k(X_n;\rho)
  \roundb{
    1+\log\roundb{1+\norm{G_k(X_n)}_{\op}/\rho}
  },
\]
given by \citet[Lemma~3]{calandriello2017second}. Our upper-bound analysis proceeds in the opposite direction, using an integral identity to derive information-gain bounds from effective-dimension bounds. The previous literature also largely focuses on the Mat\'ern and squared exponential kernels \citep{srinivas2009gaussian,iwazaki2025improved,iwazaki2026tighter}; \cref{tab:comparison} compares our rates for these kernels with the previous upper and lower bounds.

\paragraph{Bounds in expectation} \citet{seeger2008information} showed the in-expectation bounds on information gain using Mercer's decomposition of the kernel.

\begin{theorem}[\citet{mercer1909xvi}; version here adopted from \citet{sun2005mercer}]\label{thm:mercer}
 Let $\cX$ be a $\sigma$-compact metric space, let $P$ be a Borel probability measure on $\cX$ with full support. Let $k \in L^2(P \otimes P)$ be a continuous kernel on $\cX$ and define the integral operator
\[
  (T_k^P f)(x) := \int k(x,y) f(y) P(\dif y)\,, \qquad f \in L^2(P)\,.
\]
Then, there exist eigenvalues $\lambda_1, \lambda_2, \cdots \geq 0$ and an orthonormal family of continuous functions $(\phi_j)_{j \geq 1} \subset L^2(P)$ such that $T_k^P \phi_j = \lambda_j \phi_j$ for all $j \geq 1$. Moreover,
\[
    k(x,y) = \sum_{j=1}^\infty  \lambda_j \phi_j(x) \phi_j(y) \qquad \text{for every $x,y \in \cX$}\,,
\]
where the series converges absolutely and uniformly on every compact subset of $\cX \times \cX$.
\end{theorem}

The argument of \citet{seeger2008information} follows by establishing that for any sequence $X_n = (x_1, \dots, x_n)$ in the support of $P$, the following inequality holds:
\begin{equation}\label{eq:seeger}
  \log\det(I+\rho^{-1} G_k(X_n)) \leq \sum_{j=1}^\infty \log(1+ \rho^{-1}\lambda_j \sum_{t=1}^n \phi_j^2(x_t))\,.\tag{$\dagger$}
\end{equation}
From here, Jensen's inequality shows that if $X_n = (x_1, \dots, x_n) \sim P^{\otimes n}$, then
\[
  \E \log\det(I+\rho^{-1} G_k(X_n)) \leq \sum_{j=1}^\infty \log(1+ \rho^{-1}\lambda_j \sum_{t=1}^n \E \phi_j^2(x_t)) = \sum_{j=1}^\infty \log(1+ \rho^{-1}n\lambda_j )\,,
\]
with the equality using the normalisation $\E \phi_j^2(x_t) = \norm{\phi_j}_{L^2(P)}^2 = 1$ of eigenfunctions. From here, eigenvalue decay rates for squared exponential and Mat\'ern kernels under various covariate measures can be used to establish upper bounds \citep{widom1963asymptotic,widom1964asymptotic}.

\paragraph{Worst-case bounds with uniformly bounded eigenfunctions} We say that $k$ has uniformly bounded eigenfunctions (with respect to $P$) if it admits a continuous choice of eigenfunctions $(\phi_j)_{j\geq 1}$ such that
\[
  b := \sup_{j \geq 1} \norm{\phi_j}_{L^\infty(P)} < \infty\,.
\]
Under this assumption, \cref{eq:seeger} gives that for every $X_n=(x_1,\dots,x_n)$ in the support of $P$,
\[
  \log\det\roundb{I+\rho^{-1}G_k(X_n)} \leq \sum_{j=1}^\infty\log\roundb{1+\rho^{-1}nb^2\lambda_j}.
\]
\citet{vakili2021information} argue that this inequality implies tight-up-to-log-factors worst-case information gain bounds for the Mat\'ern and squared exponential kernels; however, they do not actually establish that either the Mat\'ern or squared exponential kernels satisfy the assumption of uniformly bounded eigenfunctions.\footnote{\citet{vakili2021information} cite \citet{riutort2023practical} in support of the claim that the Mat\'ern and squared exponential kernels on the hypercube admit uniformly bounded eigenfunctions. However, \citet{riutort2023practical} work with eigenfunctions of the Laplace operator on an extended domain under Dirichlet boundary conditions, and not the Mercer eigenfunctions on the original domain.}

The property that the eigenfunctions are uniformly bounded is not trivially granted. \citet{zhou2002covering} provide a counter-example of a smooth compactly supported kernel with eigenfunctions (with respect to the Lebesgue measure) growing without bound. \citet{minh2006mercer} provide a whole family of continuous counter-examples on the unit sphere in $\Rd$ for $d \geq 3$.

\begin{table}[t]
  \centering
  \caption{Comparison of bounds on $\Gamma_n^k([0,1]^d; 1)$ for the Mat\'ern-$\nu$ and squared exponential kernels as $n \to \infty$. Previous lower bound column follows from combining the results of \citet{scarlett2017lower,li2022gaussian}. Previous upper bounds are due to \citet{iwazaki2025improved,iwazaki2026tighter}, and require $\nu > 1/2$ for the Mat\'ern kernel. Rates and bounds are understood up to constants independent of $n$.}\label{tab:comparison}
  \begin{tabular}{@{}lccc@{}}
  \toprule
  Kernel
  & Sharp rate
  & Previous lower bound
  & Previous upper bound \\
  \midrule
  Mat\'ern-$\nu$
  & $n^{\frac{d}{2\nu + d}}$
  & $n^{\frac{d}{2\nu + d}}\frac{1}{\log n(\log\log n)^2 }$
  & $n^{\frac{d}{2\nu + d}}(\log n)^{\frac{4\nu + d}{2\nu + d}}$  \\
  \addlinespace[2pt]
  Squared exponential
  & $\frac{(\log n)^{d+1}}{(\log\log n)^d}$
  & $\frac{(\log n)^{\frac{d}{2}-1}}{(\log\log n)^{2}}$
  & $\frac{(\log n)^{d+1}}{(\log\log n)^d}$\\
  \bottomrule
  \end{tabular}
\end{table}

\paragraph{Bounds via the hypersphere}
\citet{iwazaki2025improved,iwazaki2026tighter} establish upper bounds on the information gain on the unit sphere in $\Rd$ for the Mat\'ern and squared exponential kernels, respectively. Their proofs expand the kernel in spherical harmonics and use the addition theorem to control the tail. They then relate the information gain on unit sphere to that on compact subsets of \(\Rd\). The resulting upper bound is tight for the squared exponential kernel, but yields an excess logarithmic factor for the Mat\'ern kernel. The proof technique is limited to zonal kernels, that is, kernels of the form $k(x,y) = g(\langle x, y \rangle)$ for a continuous function~$g$.

\paragraph{Lower bounds} \citet{scarlett2017lower} gives algorithm-independent lower bounds on the expected regret in a Gaussian process bandit setting for the Mat\'ern and squared exponential kernels. Comparing these with the regret upper bound given by \citet{li2022gaussian} yields the lower bounds on worst-case information gain for these two kernels reported in \cref{tab:comparison}. The resulting lower bound for the squared exponential kernel does not recover the correct leading exponent. This is likely due to the proof of \citet{scarlett2017lower} being based on Fourier transforms. Our approach is different as we switch to a Mercer-based construction to obtain the tight rate for kernels with super-exponential spectral decay.

\section*{Acknowledgements and Disclosure of Funding}

The research of Arya Akhavan was funded by UK Research and Innovation (UKRI) under the UK government’s Horizon Europe funding guarantee [grant number EP/Y028333/1]. The work of Alexandre B. Tsybakov was supported by Labex Ecodec (ANR-11-LABEX-0047) and by ANR MaLIP (ANR-25-CE40-3228-01).

\printbibliography

\clearpage

\appendix
\section{Proofs of upper bounds}

This appendix collects the deferred upper-bound arguments. We first record function space definitions, and the native-space facts used throughout and prove the scalar-kernel smoothness criterion. We then derive ultradifferentiable regularity from stretched-exponential moments of the spectral measure and prove the entire-extension results used for factorial residuals.

\subsection{Function-space conventions}
\label{app:functions}

We record the Bessel-potential and H\"older-space conventions used in the main text.

\paragraph{Bessel-potential spaces} Let $q\in\Np$ and $t>0$. The Bessel potential space $H^t(\R^q)$ consists of the functions $F\in L^2(\R^q)$ such that
\[
  \norm{F}_{H^t(\R^q)}^2:=\int_{\R^q}\langle\xi\rangle^{2t}\abs{\widehat F(\xi)}^2\dif\xi<\infty,
\]
where $\widehat F(\xi):=\int_{\R^q}F(x)e^{-2\pi i\xi\cdot x}\dif x$ interpreted for $F \in L^2$ by the usual extension. If $D\subset\R^q$ is open, or if $q=d$ and $D=\QQ$, define $H^t(D):=\curlyb{F|_D:F\in H^t(\R^q)}$ with norm
\[
  \norm{f}_{H^t(D)}:=\inf\curlyb{\norm{F}_{H^t(\R^q)}:F\in H^t(\R^q),\ F|_D=f}.
\]
Restrictions are understood almost everywhere.

\paragraph{H\"older spaces} Let $U\subset\R^q$ be open, let $r\in\N$, and let $\theta\in(0,1]$. As usual, $C^r(U)$ denotes the space of functions with continuous partial derivatives up to order $r$. For a compact set $K\subset U$, define
\[
  [g]_{C^{0,\theta}(K)}:=\sup_{\substack{x,y\in K\\x\neq y}}\frac{\abs{g(x)-g(y)}}{\norm{x-y}^{\theta}}.
\]
The space $C^{r,\theta}(U)$ consists of the functions $f\in C^r(U)$ such that $[\partial^\alpha f]_{C^{0,\theta}(K)}<\infty$ for every compact $K\subset U$ and every multi-index $\alpha$ with $|\alpha|=r$.

The space $C^r(\QQ)$ consists of the functions whose partial derivatives up to order $r$ on $(0,1)^d$ extend continuously to $\QQ$. For $\theta\in(0,1]$, the space $C^{r,\theta}(\QQ)$ consists of those $f\in C^r(\QQ)$ for which
\[
  \norm{f}_{C^{r,\theta}(\QQ)}:=\max_{|\alpha|\leq r}\norm{\partial^\alpha f}_{L^\infty(\QQ)}+\max_{|\alpha|=r}[\partial^\alpha f]_{C^{0,\theta}(\QQ)}
\]
is finite.

For $s>0$, write $s=r+\theta$ with $r:=\floor{s}$ and $\theta\in[0,1)$. We set $C^s:=C^r$ when $\theta=0$ and $C^s:=C^{r,\theta}$ when $\theta>0$, both on open sets and on $\QQ$. We retain the notation $C^{r,1}$ for the Lipschitz endpoint.

\subsection{Native space preliminaries}

The following three standard facts are used in the main text and the deferred proofs.
\begin{lemma}[Restriction of a native space]
\label{lem:rkhs-restriction}
  Let \(U\subset\Rd\), let \(\QQ\subset U\), and let \(\tilde k\) be a positive semidefinite kernel on \(U\). Set \(k:=\tilde k|_{\QQ\times\QQ}\). Then \[\cH_k = \curlyb{F|_{\QQ}: F\in\cH_{\tilde k}},\]
  and, for every \(f\in\cH_k\),
  \[
    \norm{f}_{\cH_k} = \min\curlyb{
      \norm{F}_{\cH_{\tilde k}} \colon F\in\cH_{\tilde k},\ F|_{\QQ}=f}.
  \]
\end{lemma}

\begin{proof}
  This is the restriction and extension theorem for native spaces;
  see \citet[Theorems~10.46--10.47]{wendland2004scattered}.
\end{proof}

\begin{lemma}[Derivative reproducing property]
\label{lem:derivative-reproducing}
  Let \(r\in\N\), let \(U\subset\Rd\) be open, and let \(k\) be a
  positive semidefinite kernel on \(U\) such that
  \(k\in C^{2r}(U\times U)\). Then every \(f\in\cH_k\) belongs to
  \(C^r(U)\), and, for every multi-index \(\alpha\) with
  \(|\alpha|\leq r\) and \(x\in U\),
  \[
    \partial_x^\alpha k(x,\cdot)\in\cH_k
    \qquad\text{and}\qquad
    \partial^\alpha f(x)
    =
    \left\langle
      f,\partial_x^\alpha k(x,\cdot)
    \right\rangle_{\cH_k}.
  \]
\end{lemma}

\begin{proof}
  This is an immediate consequence of \citet[Theorems~10.45 and~16.7]{wendland2004scattered}.
\end{proof}

\begin{lemma}
\label{lem:feature-space-representation}
  Let \(\cX\) be a set, let \(\cF\) be a Hilbert space, and let
  \(\Phi\colon\cX\to\cF\). Suppose that $k(x,y)=\langle\Phi(y),\Phi(x)\rangle_{\cF}$ for all $x,y\in\cX$. Then, for every \(f\in\cH_k\), there exists a unique vector \(v_f\)
  in the closure in \(\cF\) of the span of
  \(\curlyb{\Phi(x):x\in\cX}\) such that
  \[
    f(x)=\langle v_f,\Phi(x)\rangle_{\cF},
    \qquad
    x\in\cX,
    \qquad
    \norm{v_f}_{\cF}=\norm{f}_{\cH_k}.
  \]
\end{lemma}

\begin{proof}
  Let \(\cF_0\) be the closure in \(\cF\) of the span of
  \(\curlyb{\Phi(x):x\in\cX}\), and define
  \(Tv:=\roundb{x\mapsto\langle v,\Phi(x)\rangle_{\cF}}\).
  Then \(\ker T=\cF_0^\perp\), since \(Tv=0\) if and only if \(v\)
  is orthogonal to \(\Phi(x)\) for every \(x\in\cX\). By
  \citet[Theorem~4.21 and its proof]{steinwart2008support}, the
  restriction of \(T\) to \((\ker T)^\perp=\cF_0\) is an isometric
  isomorphism onto \(\cH_k\). The claim follows.
\end{proof}

\subsection{H\"older regularity from kernel smoothness}\label{app:holder-continuity}

We restate and prove the kernel-smoothness sufficient condition for algebraic residuals.

\CorSmoothKernelResiduals*

\begin{proof}
  Fix \(f\in\cH_k\). By \cref{lem:rkhs-restriction}, there exists \(F\in\cH_{\tilde k}\) such that \(F|_{\QQ}=f\) and \(\norm{F}_{\cH_{\tilde k}}=\norm{f}_{\cH_k}\). Write \(s=2r+\tau\), where \(r:=\lfloor s/2\rfloor\) and \(\tau\in[0,2)\). For \(|\alpha|\leq r\), set \(\psi_\alpha(x):=\partial_x^\alpha\tilde k(x,\cdot)\) and \(g_\alpha(x,y):=\partial_x^\alpha\partial_y^\alpha\tilde k(x,y)\). Since \(\tilde k\in C^{2r}(U\times U)\), \cref{lem:derivative-reproducing} gives
  \[
    \partial^\alpha F(x)=\langle F,\psi_\alpha(x)\rangle_{\cH_{\tilde k}},\qquad \langle\psi_\alpha(x),\psi_\alpha(y)\rangle_{\cH_{\tilde k}}=g_\alpha(x,y).
  \]
  The second identity follows by applying the derivative reproducing formula to \(\psi_\alpha(x)\). In particular, \(\norm{\psi_\alpha(x)}_{\cH_{\tilde k}}^2=g_\alpha(x,x)\). Since \(x\mapsto g_\alpha(x,x)\) is continuous on the compact set \(\QQ\), Cauchy--Schwarz gives
  \[
    \norm{\partial^\alpha F}_{L^\infty(\QQ)}\leq C\norm{F}_{\cH_{\tilde k}},\qquad |\alpha|\leq r.
  \]
  If \(\tau=0\), then \(s/2=r\), and the required embedding follows.

  Suppose that \(0<\tau<2\), and fix \(|\alpha|=r\). Then \(g_\alpha\in C^\tau(U\times U)\), and symmetry of \(\tilde k\) gives \(g_\alpha(x,y)=g_\alpha(y,x)\). Moreover,
  \[
    \norm{\psi_\alpha(x)-\psi_\alpha(y)}_{\cH_{\tilde k}}^2=g_\alpha(x,x)+g_\alpha(y,y)-2g_\alpha(x,y).
  \]
  If \(0<\tau\leq1\), symmetry yields
  \[
    g_\alpha(x,x)+g_\alpha(y,y)-2g_\alpha(x,y)=\roundb{g_\alpha(x,x)-g_\alpha(x,y)}+\roundb{g_\alpha(y,y)-g_\alpha(y,x)}.
  \]
  H\"older continuity when \(\tau<1\), and the mean value theorem when \(\tau=1\), therefore give
  \[
    \norm{\psi_\alpha(x)-\psi_\alpha(y)}_{\cH_{\tilde k}}^2\leq C\norm{x-y}^{\tau},\qquad x,y\in\QQ.
  \]

  Now suppose that \(1<\tau<2\), and write \(\beta:=\tau-1\). For \(x,y\in\QQ\), put \(h:=y-x\), and write \(\nabla_u g_\alpha(u,v)\) for the gradient with respect to the first argument. Since \(\QQ\) is convex, the fundamental theorem of calculus and symmetry give
  \[
    g_\alpha(x,x)+g_\alpha(y,y)-2g_\alpha(x,y)=\int_0^1\left\langle\nabla_u g_\alpha(x+th,y)-\nabla_u g_\alpha(x+th,x),h\right\rangle\dif t.
  \]
  The map \(\nabla_u g_\alpha\) is \(\beta\)-H\"older on \(\QQ\times\QQ\), so the integrand is bounded in absolute value by \(C\norm{x-y}^{1+\beta}=C\norm{x-y}^{\tau}\). Hence the same bound on \(\norm{\psi_\alpha(x)-\psi_\alpha(y)}_{\cH_{\tilde k}}^2\) holds in this case.

  We have therefore shown that, for every \(0<\tau<2\),
  \[
    \norm{\psi_\alpha(x)-\psi_\alpha(y)}_{\cH_{\tilde k}}\leq C\norm{x-y}^{\tau/2},\qquad x,y\in\QQ,\quad |\alpha|=r.
  \]
  By Cauchy--Schwarz,
  \[
    \abs{\partial^\alpha F(x)-\partial^\alpha F(y)}\leq C\norm{F}_{\cH_{\tilde k}}\norm{x-y}^{\tau/2},\qquad x,y\in\QQ,\quad |\alpha|=r.
  \]
  Combining this estimate with the uniform derivative bounds gives
  \[
    \norm{f}_{C^{r,\tau/2}(\QQ)}\leq C\norm{f}_{\cH_k}.
  \]
  Since \(r+\tau/2=s/2\), this proves \(\cH_k\hookrightarrow C^{s/2}(\QQ)\). By \cref{prop:finite-smoothness-residuals}, \(\epsilon_M\lesssim M^{-s/2}\). Finally, \(x\mapsto\tilde k(x,x)\) is continuous on the compact set \(\QQ\), and hence \(\kappa_0<\infty\).
\end{proof}

\subsection{Exponential moments and ultradifferentiability}\label{app:stretched-exp-moments}

We prove that a stretched-exponential moment of the spectral measure yields stretched-exponential residuals.

\PropExpMomentGev*

\begin{proof}
  Since \(\mu\) is symmetric, \(K\) is real-valued. Moreover, for \(N\in\Np\), \(x_1,\ldots,x_N\in\Rd\), and \(a_1,\ldots,a_N\in\R\),
  \[
    \sum_{i,j=1}^N a_i a_j\tk(x_i,x_j)=\int_{\Rd}\abs{\sum_{i=1}^N a_i e^{2\pi i \xi \cdot x_i}}^2\mu(\dif\xi)\geq0.
  \]
  Thus \(\tk\) is positive semidefinite.

  Set \(M_0:=\int_{\Rd}\exp\{\tau_0\norm{\xi}^{\gamma}\}\mu(\dif\xi)<\infty\). For \(m\in\Np\), the function \(r\mapsto r^m e^{-\tau_0r^\gamma}\) attains its maximum at \(r^\gamma=m/(\gamma\tau_0)\), and hence $r^m\leq\roundb[\big]{\frac{m}{e\gamma\tau_0}}^{m/\gamma}e^{\tau_0r^\gamma}$ for $r \geq 0$. It follows that
  \[
    \int_{\Rd}\norm{\xi}^m\mu(\dif\xi)\leq M_0\roundb[\Big]{\frac{m}{e\gamma\tau_0}}^{m/\gamma}\leq M_0(\gamma\tau_0)^{-m/\gamma}(m!)^{1/\gamma},
  \]
  where the final inequality uses \(m!\geq(m/e)^m\). After increasing the constant to include \(m=0\), there exists \(B\geq1\) such that
  \[
    \int_{\Rd}\norm{\xi}^m\mu(\dif\xi)\leq B^{m+1}(m!)^{1/\gamma},\qquad m\in\N.
  \]

  These moment bounds provide an integrable majorant for every derivative of the Fourier integrand. Dominated convergence therefore shows that \(K\in C^\infty(\Rd)\) and, for every \(\eta\in\Nd\),
  \[
    \partial^\eta K(h)=\int_{\Rd}(2\pi i\xi)^\eta e^{2\pi i \xi \cdot h}\mu(\dif\xi).
  \]
  Writing \(m:=|\eta|\), we obtain \(\abs{\partial^\eta K(h)}\leq(2\pi)^mB^{m+1}(m!)^{1/\gamma}\). Since $\frac{m!}{\eta!}=\binom{m}{\eta_1,\ldots,\eta_d}\leq d^m$, increasing \(B\) gives
  \[
    \abs{\partial^\eta K(h)}\leq B^{|\eta|+1}(\eta!)^{1/\gamma},\qquad h\in\Rd,\ \eta\in\Nd.
  \]

  For \(\alpha,\beta\in\Nd\), differentiation of \(\tk(x,y)=K(x-y)\) gives \(\partial_x^\alpha\partial_y^\beta\tk(x,y)=(-1)^{|\beta|}\partial^{\alpha+\beta}K(x-y)\). Moreover, \((\alpha_j+\beta_j)!\leq2^{\alpha_j+\beta_j}\alpha_j!\beta_j!\) for every \(j\in[d]\), and hence \((\alpha+\beta)!\leq2^{|\alpha|+|\beta|}\alpha!\beta!\). Increasing \(B\) once more yields
  \[
    \abs{\partial_x^\alpha\partial_y^\beta\tk(x,y)}\leq B^{|\alpha|+|\beta|+1}(\alpha!)^{1/\gamma}(\beta!)^{1/\gamma}
  \]
  for all \(x,y\in\Rd\) and \(\alpha,\beta\in\Nd\). Thus \(\tk\in G^{1/\gamma}(\Rd\times\Rd)\), and \cref{cor:gevrey-kernel-residual} gives $\epsilon_M\lesssim\exp\curlyb{-cM^\gamma}$.

  Finally, suppose that \(\mu(\dif\xi)=\exp\{-\tau\norm{\xi}^{\gamma}\}\dif\xi\). Then
  \[
    \int_{\Rd}\exp\curlyb[\Big]{\frac{\tau}{2}\norm{\xi}^{\gamma}}\mu(\dif\xi)=\int_{\Rd}\exp\curlyb[\Big]{-\frac{\tau}{2}\norm{\xi}^{\gamma}}\dif\xi<\infty,
  \]
  where finiteness follows by polar coordinates. Moreover, \(k(x,x)=K(0)=\mu(\Rd)<\infty\). The two complexity bounds follow from \cref{cor:canonical-residual-profiles} with \(\sigma=1/\gamma\).
\end{proof}

\subsection{Entire extensions}\label{app:entire-moments}

We first record the tensor-product Chebyshev expansion used to construct the positive semidefinite majorant.

\begin{lemma}\label{lem:analytic-chebyshev}
  Let \(q\in\Np\), let \(\varrho>1\), and let
  \(f\colon\bC^q\to\bC\) be holomorphic on an open neighbourhood of
  \(\EE_\varrho^q\). Let \(T_0,T_1,\dots\) be the Chebyshev
  polynomials of the first kind, and set
  \[
    T_\alpha(x)
    :=
    \prod_{j=1}^q T_{\alpha_j}(x_j),
    \qquad
    \alpha\in\N^q.
  \]
  Then
  \[
    f(x)
    =
    \sum_{\alpha\in\N^q}a_\alpha T_\alpha(x),
    \qquad
    x\in[-1,1]^q,
  \]
  with absolute and uniform convergence. Moreover,
  \[
    \abs{a_\alpha}
    \leq
    2^q
    \sup_{z\in\EE_\varrho^q}\abs{f(z)}
    \varrho^{-|\alpha|},
    \qquad
    \alpha\in\N^q.
  \]
\end{lemma}

\begin{proof}
This is \citet[Corollary~3.2]{wang2020analysis}; see also \citet[Lemma~5.1]{trefethen2017multivariate}.
\end{proof}

\PropEntireMajorants*

\begin{proof}
  Fix \(M\in\Np\) and \(\varrho\geq2\). Applying
  \cref{lem:analytic-chebyshev} in dimension \(2d\) to \(\tk\) gives
  an absolutely and uniformly convergent expansion
  \[
    k(x,y)
    =
    \sum_{\alpha,\beta\in\Nd}
    a_{\alpha,\beta}T_\alpha(x)T_\beta(y),
    \qquad
    x,y\in\QQpm,
  \]
  such that
  \[
    \abs{a_{\alpha,\beta}}
    \leq
    2^{2d}B_{\tk}(\varrho)
    \varrho^{-|\alpha|-|\beta|}.
  \]
  For \(\alpha\in\Nd\), define
  \[
    b_\alpha
    :=
    \frac12
    \sum_{\beta\in\Nd}
    \roundb{
      \abs{a_{\alpha,\beta}}
      +
      \abs{a_{\beta,\alpha}}
    }.
  \]
  The coefficient estimate implies that
  \(\sum_{\alpha\in\Nd}b_\alpha<\infty\). Hence
  \[
    H_\varrho(x,y)
    :=
    \sum_{\alpha\in\Nd}
    b_\alpha T_\alpha(x)T_\alpha(y)
  \]
  defines a positive semidefinite kernel on \(\QQpm\). We claim that \(H_\varrho\) majorises \(k\). Fix \(n\in\Np\), \(X_n=(x_1,\dots,x_n)\in(\QQpm)^n\), and \(v\in\R^n\). For \(\alpha\in\Nd\), write $q_\alpha := \sum_{i=1}^n v_iT_\alpha(x_i)$. Absolute convergence of the Chebyshev expansion gives
  \[
    v\tran G_k(X_n)v = |\sum_{\alpha,\beta\in\Nd} a_{\alpha,\beta}q_\alpha q_\beta|
      \leq \frac12 \sum_{\alpha,\beta\in\Nd} \abs{a_{\alpha,\beta}} \roundb{q_\alpha^2+q_\beta^2}
      = \sum_{\alpha\in\Nd}b_\alpha q_\alpha^2
      = v\tran G_{H_\varrho}(X_n)v.
  \]
  Therefore $G_k(X_n)\preceq G_{H_\varrho}(X_n)$.

  Let $\cI_M := \curlyb{0,\dots,M-1}^d$ and define
  \[
      K_{M,\varrho}(x,y) :=\sum_{\alpha\in\cI_M} b_\alpha T_\alpha(x)T_\alpha(y),\qquad  R_{M,\varrho}(x,y) := \sum_{\alpha\notin\cI_M} b_\alpha T_\alpha(x)T_\alpha(y).
  \]
  Both kernels are positive semidefinite, and
  \(H_\varrho=K_{M,\varrho}+R_{M,\varrho}\). It follows that
  \[
    G_k(X_n)
    \preceq
    G_{K_{M,\varrho}}(X_n)
    +
    G_{R_{M,\varrho}}(X_n).
  \]
  Since \(K_{M,\varrho}\) is a sum of \(M^d\) rank-one kernels, $\rank G_{K_{M,\varrho}}(X_n) \leq M^d$.

  It remains to bound the residual diagonal. Since
  \(\varrho\geq2\),
  \[
    \sum_{\beta\in\Nd}\varrho^{-|\beta|}  = \roundb[\Big]{\frac{\varrho}{\varrho-1}}^d \leq 2^d.
  \]
  Hence there exists a constant \(C_d>0\), depending only on \(d\),
  such that
  \[
    b_\alpha
    \leq
    C_d B_{\tk}(\varrho)\varrho^{-|\alpha|},
    \qquad
    \alpha\in\Nd.
  \]
  Since \(\abs{T_\alpha(x)}\leq1\) on \(\QQpm\),
  \[
    R_{M,\varrho}(x,x)
    \leq
    C_d B_{\tk}(\varrho)
    \sum_{\alpha\notin\cI_M}\varrho^{-|\alpha|}.
  \]
  Every \(\alpha\notin\cI_M\) has \(\alpha_j\geq M\) for some
  \(j\in[d]\). Therefore
  \[
    \sum_{\alpha\notin\cI_M}\varrho^{-|\alpha|} \leq d\roundb[\bigg]{\sum_{m=M}^\infty\varrho^{-m}}\roundb[\bigg]{\sum_{m=0}^\infty\varrho^{-m}}^{d-1}
    \leq 2^d d \varrho^{-M}.
  \]
  Increasing \(C_d\) gives
  \[
    \sup_{x\in\QQpm}R_{M,\varrho}(x,x)
    \leq
    C_dB_{\tk}(\varrho)\varrho^{-M},
  \]
  as required.
\end{proof}

We next verify the finite-order condition for stationary kernels from a super-exponential moment of the spectral measure.

\PropStatFiniteOrder*

\begin{proof}
  By Young's inequality, there exists \(c_0>0\) such that
  \[
    2\pi rs
    \leq
    \frac{\tau_0}{2}s^\gamma+c_0r^p,
    \qquad
    r,s\geq0.
  \]
  Define
  \[
    \tk(z,w)
    :=
    \int_{\Rd}
    \exp\curlyb{2\pi i\xi\cdot(z-w)}
    \mu(\dif\xi),
    \qquad
    z,w\in\bC^d.
  \]
  For real \(x,y\in\Rd\), this agrees with \(k(x,y)\).

  We first verify that \(\tk\) is entire. Let
  \(L\subset\bC^d\times\bC^d\) be compact, and set
  \[
    R_L
    :=
    \sup_{(z,w)\in L}
    \norm{\Im(z-w)}.
  \]
  For multi-indices \(\alpha,\beta\in\Nd\), differentiating the
  integrand formally gives
  \[
    \partial_z^\alpha\partial_w^\beta
    \exp\curlyb{2\pi i\xi\cdot(z-w)}
    =
    (2\pi i\xi)^\alpha
    (-2\pi i\xi)^\beta
    \exp\curlyb{2\pi i\xi\cdot(z-w)}.
  \]
  Its absolute value is at most
  \[
    (2\pi)^{|\alpha|+|\beta|}
    \norm{\xi}^{|\alpha|+|\beta|}
    \exp\curlyb{2\pi R_L\norm{\xi}}.
  \]
  The preceding form of Young's inequality gives
  \[
    \exp\curlyb{2\pi R_L\norm{\xi}}
    \leq
    \exp\curlyb{c_0R_L^p}
    \exp\curlyb[\Big]{\frac{\tau_0}{2}\norm{\xi}^{\gamma}}.
  \]
  Moreover, for every \(m\in\N\), there exists \(C_m>0\) such that
  \[
    r^m
    \leq
    C_m
    \exp\curlyb[\Big]{\frac{\tau_0}{2}r^\gamma},
    \qquad
    r\geq0.
  \]
  Thus every derivative of the integrand is bounded on \(L\) by a
  constant multiple of $\exp\curlyb{\tau_0\norm{\xi}^{\gamma}}$, which is integrable with respect to \(\mu\). Dominated convergence
  therefore permits differentiation under the integral in every
  complex coordinate. Hence \(\tk\) is entire.

  The same Young inequality gives, for all \(z,w\in\bC^d\),
  \begin{align*}
    \abs{\tk(z,w)}
    &\leq \int_{\Rd}\exp\curlyb{2\pi\norm{\xi}\norm{\Im(z-w)}}\mu(\dif\xi)\\
    &\leq \exp\curlyb{c_0\norm{\Im(z-w)}^p} \int_{\Rd}\exp\curlyb[\Big]{\frac{\tau_0}{2}\norm{\xi}^{\gamma}}\mu(\dif\xi).
  \end{align*}
  The integral on the right is finite. Since
  \[
    \norm{\Im(z-w)}
    \leq
    2\sqrt d\,
    \max\roundb{\norm{z}_\infty,\norm{w}_\infty},
  \]
  there exist \(C_1,c_1>0\) such that
  \[
    \abs{\tk(z,w)}
    \leq
    C_1
    \exp\curlyb{
      c_1
      \max\roundb{
        \norm{z}_\infty^p,
        \norm{w}_\infty^p
      }
    },
    \qquad
    z,w\in\bC^d.
  \]
  Thus the restriction of \(k\) to
  \(\QQpm\times\QQpm\) is entire of order at most \(p\). The residual
  bound follows from
  \cref{cor:finite-order-entire-residual}.
\end{proof}

\clearpage

\section{Proofs of lower bounds}
\label{sec:lower-bounds}

This appendix proves the three lower-bound theorems. The algebraic and stretched-exponential proofs use lattice Fourier directions, while the super-exponential proof uses Mercer directions. Each argument first treats $a:=n/\rho$ above a fixed constant; the repeated-point construction below extends the bounds to the full range $1\leq a\leq n$.

\begin{lemma}[Repeated-point design]
\label{lem:repeated-point-lower}
  Let \(k\) be a positive semidefinite kernel on a set \(\cX\), and suppose that \(k(x_0,x_0)=\kappa>0\) for some \(x_0\in\cX\). Then, for every \(n\in\Np\) and \(0<\rho\leq n\),
  \[
    \Effd_n^k(\cX;\rho)\geq\frac{\kappa}{1+\kappa}\spaced{and}\Gamma_n^k(\cX;\rho)\geq\log(1+\kappa).
  \]
\end{lemma}

\begin{proof}
  Take \(X_n=(x_0,\dots,x_0)\). Then \(G_k(X_n)=\kappa\one\one\tran\), whose only nonzero eigenvalue is \(n\kappa\). Hence
  \[
    \Effd_n^k(\cX;\rho)\geq\frac{n\kappa}{n\kappa+\rho}\geq\frac{\kappa}{1+\kappa},\qquad \Gamma_n^k(\cX;\rho)\geq\log\roundb[\Big]{1+\frac{n\kappa}{\rho}}\geq\log(1+\kappa). \qedhere
  \]
\end{proof}

\begin{remark}
\label{rem:extend-lower-bounds}
Let $a_0\geq1$, and let $q_1,q_2\colon[1,\infty)\to[0,\infty)$ be bounded on $[1,a_0]$. Suppose that, for every $n\in\Np$ and $0<\rho\leq n$ with $a:=n/\rho\in[a_0,n]$,
\[
  \Effd_n^k(\cX;\rho)\geq Cq_1(a)\spaced{and}\Gamma_n^k(\cX;\rho)\geq Cq_2(a).
\]
Each spectral lower bound considered below implies $k(0,0)>0$, and therefore \cref{lem:repeated-point-lower} extends both bounds to $a\in[1,n]$ after possibly decreasing~$C$.
\end{remark}

\subsection{Lattice Fourier preliminaries}\label{sec:lattice-fourier-lower}

We record the Fourier convention, Poisson summation formula, and lattice compression used in the algebraic and stretched-exponential proofs.

For \(f,g\colon\Rd\to\bC\), write \((f*g)(x):=\int_{\Rd}f(x-y)g(y)\dif y\).  We use the Fourier transform convention
\[
  \hat f(\xi):=\int_{\Rd}f(x)e^{-2\pi i \xi \cdot x}\dif x.
\]
With this convention, \(\widehat{f*g}=\hat f\,\hat g\) and \(\widehat{fg}=\hat f*\hat g\) whenever both sides are defined.

\begin{proposition}[Poisson summation]
\label{prop:poisson-summation}
  Let \(f\colon\Rd\to\bC\) be continuous, and suppose that, for some \(p>d\),
  \[
    \abs{f(x)}\lesssim\langle x\rangle^{-p}\spaced{and}\abs{\hat f(\xi)}\lesssim\langle\xi\rangle^{-p},\qquad x,\xi\in\Rd.
  \]
  Then, for every \(T>0\),
  \[
    \sum_{\ell\in\Zd}f(x+T\ell)=\sum_{m\in\Zd}T^{-d}\hat f(m/T)e^{2\pi i (m/T) \cdot x},\qquad x\in\Rd.
  \]
  Both series converge absolutely and uniformly on \([0,T]^d\).
\end{proposition}

\begin{proof}
  Apply \citet[Theorem~3.2.8]{grafakos2008classical} to \(g(y):=f(Ty)\) and evaluate the resulting identity at \(y=x/T\).
\end{proof}

We next identify Fourier directions on the uniform lattice along which the periodised Gram matrix is diagonal. For \(L\in\Np\), define
\[
  \Lambda_L:=L^{-1}\curlyb{0,\dots,L-1}^d\subset\QQ,\qquad n_0:=\card(\Lambda_L)=L^d.
\]
For \(m\in\curlyb{0,\dots,L-1}^d\), define
\[
  u_m:=n_0^{-1/2}\roundb{e^{2\pi i m \cdot x}}_{x\in\Lambda_L}\in\bC^{n_0}.
\]
For \(\cM\subset\curlyb{0,\dots,L-1}^d\), let \(U_\cM:=\squareb{u_m}_{m\in\cM}\).

\begin{proposition}
\label{prop:lattice-fourier-directions}
  Let $L,T\in\Np$ and let $\cM\subset\curlyb{0,\dots,L-1}^d$. Let \(K\) be a stationary kernel on \(\Rd\) with spectral density \(S\), and suppose that its \(T\)-periodisation
  \[
    K_T^\#(x):=\sum_{\ell\in\Zd}K(x+T\ell)
  \]
  has the absolutely convergent Fourier expansion
  \[
    K_T^\#(x)=\sum_{r\in\Zd}T^{-d}S(r/T)e^{2\pi i (r/T) \cdot x}.
  \]
  Set \(E_{L,T}:=G_K(\Lambda_L)-G_{K_T^\#}(\Lambda_L)\). Then \(U_\cM^*U_\cM=I\), and
  \[
    U_\cM^*G_K(\Lambda_L)U_\cM\succeq\diag\roundb{n_0T^{-d}S(m):m\in\cM}-\norm{E_{L,T}}_{\op}I.
  \]
\end{proposition}

\begin{proof}
  For \(m,m'\in\curlyb{0,\dots,L-1}^d\),
  \[
    u_m^*u_{m'}=\prod_{j=1}^d\frac1L\sum_{\ell=0}^{L-1}e^{2\pi i(m_j'-m_j)\ell/L}=\1{m=m'}.
  \]
  Hence \(U_\cM^*U_\cM=I\). For \(r\in\Zd\), let \(v_r:=\roundb{e^{2\pi i (r/T) \cdot x}}_{x\in\Lambda_L}\). The Fourier expansion gives
  \[
    G_{K_T^\#}(\Lambda_L)=\sum_{r\in\Zd}T^{-d}S(r/T)v_rv_r^*\succeq\sum_{m\in\cM}T^{-d}S(m)v_{Tm}v_{Tm}^*.
  \]
  Since \(v_{Tm}=n_0^{1/2}u_m\), compression by \(U_\cM\) yields
  \[
    U_\cM^*G_{K_T^\#}(\Lambda_L)U_\cM\succeq\diag\roundb{n_0T^{-d}S(m):m\in\cM}.
  \]
  Finally, \(U_\cM^*E_{L,T}U_\cM\succeq-\norm{E_{L,T}}_{\op}I\), which proves the result.
\end{proof}

\subsection{Algebraic spectral decay}\label{app:algebraic-lb}

We first construct a witness kernel with the required spectral decay and compact spatial support; compact support makes periodisation exact on the lattice.

\begin{lemma}[Compactly supported algebraic witness]
\label{lem:algebraic-witness}
  Let \(r>d\). There exists a stationary positive semidefinite kernel \(K\) on \(\Rd\), with spectral density \(S\), such that
  \[
    S(\xi)\asymp\langle\xi\rangle^{-r},\qquad \xi\in\Rd,
  \]
  and \(\supp K\subset(-1,1)^d\).
\end{lemma}

\begin{proof}
  Choose a nonzero, nonnegative function \(w\in C_c^\infty(\Rd)\) with \(\supp w\subset(-1/2,1/2)^d\), and define
  \[
    \widetilde w(x):=w(-x),\qquad \eta:=\frac{w*\widetilde w}{\roundb{\int_{\Rd}w(x)\dif x}^2}.
  \]
  Then \(\eta\in C_c^\infty(\Rd)\), \(\supp\eta\subset(-1,1)^d\), and
  \[
    \hat\eta(\xi)=\frac{\abs{\hat w(\xi)}^2}{\roundb{\int_{\Rd}w(x)\dif x}^2}\geq0,\qquad \hat\eta(0)=1.
  \]
  Let \(S_0(\xi):=\langle\xi\rangle^{-r}\) and \(S:=S_0*\hat\eta\). Then \(S\geq0\) and \(S\in L^1(\Rd)\). If \(K_0\) and \(K\) are the stationary kernels generated by \(S_0\) and \(S\), respectively, then \(\widehat{\eta K_0}=\hat\eta*S_0=S\), and hence \(K=\eta K_0\). Thus \(\supp K\subset(-1,1)^d\).

  Peetre's inequality \citep[Equation~(2.21)]{treves1980introduction} gives
  \[
    S(\xi)\leq2^{r/2}\langle\xi\rangle^{-r}\int_{\Rd}\hat\eta(\omega)\langle\omega\rangle^r\dif\omega\lesssim\langle\xi\rangle^{-r}.
  \]
  Since \(\hat\eta\) is continuous and \(\hat\eta(0)=1\), there exist \(c,\delta>0\) such that \(\hat\eta(\omega)\geq c\) whenever \(\norm{\omega}\leq\delta\). Moreover, \(\langle\xi-\omega\rangle\leq C_\delta\langle\xi\rangle\) on this ball. Therefore
  \[
    S(\xi)\geq c\int_{\norm{\omega}\leq\delta}\langle\xi-\omega\rangle^{-r}\dif\omega\gtrsim\langle\xi\rangle^{-r}.\qedhere
  \]
\end{proof}

With this witness, the lattice construction gives the algebraic lower bound.

\ThmPolyLB*

\begin{proof}
  Write \(a:=n/\rho\), and choose \(a_0\geq1\), depending only on \(d,r,c_0\), sufficiently large for the estimates below. We first assume that \(a\geq a_0\).

  Let \(K\) and \(S\) be given by \cref{lem:algebraic-witness}. Since \(s(\xi)\gtrsim\langle\xi\rangle^{-r}\asymp S(\xi)\), \cref{lem:spectral-domination} gives
  \[
    G_k(X)\succeq cG_K(X)
  \]
  for every finite design \(X\).

  Let \(L:=\floor{n^{1/d}}\), \(n_0:=L^d\), and \(b:=n_0/\rho\). Then \(2^{-d}a\leq b\leq a\). Since \(\supp K\subset(-1,1)^d\), the \(2\)-periodisation \(K_2^\#(x):=\sum_{\ell\in\Zd}K(x+2\ell)\) satisfies \(G_K(\Lambda_L)=G_{K_2^\#}(\Lambda_L)\). By \cref{prop:poisson-summation},
  \[
    K_2^\#(x)=\sum_{q\in\Zd}2^{-d}S(q/2)e^{2\pi i (q/2) \cdot x},
  \]
  with absolute and uniform convergence.

  Fix \(A\in(0,1]\), to be chosen below, and set \(M:=\floor{Ab^{1/r}}\) and \(\cM:=\curlyb{0,\dots,M-1}^d\). Increasing \(a_0\) if necessary gives \(M\geq1\). Since \(b\leq n_0=L^d\) and \(r>d\), \(M\leq b^{1/r}\leq L^{d/r}\leq L\).

  For \(m\in\cM\), \(S(m)\gtrsim M^{-r}\). Applying \cref{prop:lattice-fourier-directions} with \(T=2\) and \(E_{L,2}=0\) gives
  \[
    \rho^{-1}U_\cM^*G_k(\Lambda_L)U_\cM\succeq cbM^{-r}I\succeq cA^{-r}I.
  \]
  Choose \(A>0\) sufficiently small that the final matrix is bounded below by \(I\). Append arbitrary points of \(\QQ\) to \(\Lambda_L\) to obtain \(X_n\in(\QQ)^n\), and extend \(U_\cM\) by zero rows to \(\widetilde U_\cM\in\bC^{n\times M^d}\). Then \(\widetilde U_\cM^*\widetilde U_\cM=I\) and
  \[
    \widetilde U_\cM^*G_k(X_n)\widetilde U_\cM\succeq\rho I.
  \]
  By \cref{prop:resolved-directions},
  \[
    \Effd_n^k(\QQ;\rho)\geq\frac12M^d\spaced{and}\Gamma_n^k(\QQ;\rho)\geq(\log2)M^d.
  \]
  Since \(M^d\asymp b^{d/r}\asymp a^{d/r}\), this proves the result when \(a\geq a_0\). Finally, \cref{rem:extend-lower-bounds} extends these bounds from \(a\in[a_0,n]\) to \(a\in[1,n]\).
\end{proof}

\subsection{Stretched-exponential spectral decay}\label{app:sub-exp-lb}

Here periodisation is no longer exact. The next lemma gives the spatial decay needed to control the periodisation error.

\begin{lemma}[name={},restate=exponentialDualLemma]
\label{lem:stretched-exp-dual}
  Let \(\tau>0\), let \(0<\gamma\leq1\), and let \(k\colon\Rd\to\R\) be the kernel generated by
  \[
    s(\xi)=e^{-\tau\norm{\xi}^{\gamma}}.
  \]
  Then
  \[
    \abs{k(x)}\lesssim\langle x\rangle^{-(d+\gamma)},\qquad x\in\Rd.
  \]
  In particular, both \(k\) and \(s\) satisfy the conditions of \cref{prop:poisson-summation}.
\end{lemma}

\begin{proof}
  Let \(s_1(\xi):=e^{-\norm{\xi}^{\gamma}}\), and let \(k_1\) be the kernel generated by \(s_1\). Since \(s(\xi)=s_1(\tau^{1/\gamma}\xi)\), the scaling property of the Fourier transform gives
  \[
    k(x)=\tau^{-d/\gamma}k_1(\tau^{-1/\gamma}x).
  \]
  By \citet[Theorem~2.1]{blumenthal1960some},
  \[
    \lim_{\norm{x}\to\infty}\norm{x}^{d+\gamma}k_1(x)=C_{d,\gamma}
  \]
  for a finite constant \(C_{d,\gamma}\). Hence \(\abs{k_1(x)}\lesssim\langle x\rangle^{-(d+\gamma)}\) for large \(\norm{x}\), while for bounded \(\norm{x}\),
  \[
    \abs{k_1(x)}\leq\int_{\Rd}e^{-\norm{\xi}^{\gamma}}\dif\xi<\infty.
  \]
  The scaling identity gives the claimed bound. Finally, for every \(N>0\), \(s(\xi)\lesssim_N\langle\xi\rangle^{-N}\). Taking \(N>d\) proves the final claim.
\end{proof}

We now combine this decay estimate with the lattice Fourier construction.

\ThmGenExpLB*

\begin{proof}
  Write \(a:=n/\rho\), and choose \(a_0\geq1\), depending only on \(d,\gamma,\tau,c_0\), sufficiently large for the estimates below. We first assume that \(a\geq a_0\).

  Let \(S(\xi):=e^{-\tau\norm{\xi}^{\gamma}}\), and let \(K\) be the corresponding stationary kernel. By \cref{lem:spectral-domination},
  \[
    G_k(X)\succeq c_0G_K(X)
  \]
  for every finite design \(X\).

  Let \(L:=\floor{n^{1/d}}\), \(n_0:=L^d\), and \(b:=n_0/\rho\). Then \(2^{-d}a\leq b\leq a\). Set
  \[
    p:=d+\gamma,\qquad \alpha:=\frac12\roundb[\Big]{\frac1p+\frac1d},\qquad \beta:=\frac{1-d\alpha}{2}.
  \]
  Then \(\alpha p>1\), \(d\alpha<1\), and \(\beta>0\). Let \(T_0:=\ceil{2\sqrt d}\), set \(T:=\max\curlyb{T_0,\ceil{b^\alpha}}\), and define \(K_T^\#(x):=\sum_{\ell\in\Zd}K(x+T\ell)\). By \cref{lem:stretched-exp-dual,prop:poisson-summation},
  \[
    K_T^\#(x)=\sum_{q\in\Zd}T^{-d}S(q/T)e^{2\pi i (q/T) \cdot x},
  \]
  with absolute and uniform convergence.

  Set \(E:=G_K(\Lambda_L)-G_{K_T^\#}(\Lambda_L)\). If \(x_i,x_j\in\Lambda_L\), then \(x_i-x_j\in[-1,1]^d\), and \cref{lem:stretched-exp-dual} gives
  \[
    \abs{E_{ij}}\lesssim T^{-p}\sum_{\ell\neq0}\norm{\ell}^{-p}\lesssim T^{-p}.
  \]
  Consequently,
  \[
    \rho^{-1}\norm{E}_{\op}\lesssim bT^{-p}\lesssim b^{1-\alpha p}.
  \]

  Define
  \[
    M:=\floor[\Bigg]{\roundb{\frac{\beta\log b}{\tau d^{\gamma/2}}}^{1/\gamma}},\qquad \cM:=\curlyb{0,\dots,M-1}^d.
  \]
  Increasing \(a_0\) if necessary gives \(1\leq M\leq L\). For \(m\in\cM\), \(\norm{m}\leq\sqrt d\,M\), and hence \(\tau\norm{m}^{\gamma}\leq\beta\log b\). Since \(T\lesssim b^\alpha\),
  \[
    \rho^{-1}n_0T^{-d}S(m)=bT^{-d}e^{-\tau\norm{m}^{\gamma}}\gtrsim b^{1-d\alpha-\beta}=b^\beta.
  \]
  Applying \cref{prop:lattice-fourier-directions} and then \cref{lem:spectral-domination} gives
  \[
    \rho^{-1}U_\cM^*G_k(\Lambda_L)U_\cM\succeq\roundb{cb^\beta-Cb^{1-\alpha p}}I\succeq cb^\beta I,
  \]
  where the final inequality holds after increasing \(a_0\).

  Append arbitrary points of \(\QQ\) to obtain \(X_n\in(\QQ)^n\), and extend \(U_\cM\) by zero rows to \(\widetilde U_\cM\in\bC^{n\times M^d}\). By \cref{prop:resolved-directions},
  \[
    \Effd_n^k(\QQ;\rho)\geq M^d\frac{cb^\beta}{1+cb^\beta},\qquad \Gamma_n^k(\QQ;\rho)\geq M^d\log\roundb{1+cb^\beta}.
  \]
  For \(a\geq a_0\),
  \[
    M^d\asymp\roundb{\log b}^{d/\gamma}\asymp\roundb{\log(e+a)}^{d/\gamma},\qquad \frac{cb^\beta}{1+cb^\beta}\gtrsim1,\qquad \log\roundb{1+cb^\beta}\gtrsim\log(e+a).
  \]
  This proves the two bounds when \(a\geq a_0\). Finally, \cref{rem:extend-lower-bounds} extends these bounds from \(a\in[a_0,n]\) to \(a\in[1,n]\).
\end{proof}

\subsection{Mercer directions}\label{sec:mercer-directions}

The super-exponential argument uses Mercer directions in place of lattice Fourier directions. The following proposition shows that, when the $M$-th Mercer eigenvalue is sufficiently large, one can choose a design whose Gram matrix is bounded below on an $M$-dimensional subspace.

\begin{proposition}[Mercer directions]
\label{prop:mercer-directions}
  Let \(k\) be a continuous real-valued positive semidefinite kernel on \(\QQ\), let \(\kappa_0:=\sup_{x\in\QQ}k(x,x)\), and let \(\lambda_1\geq\lambda_2\geq\cdots\geq0\) be its Mercer eigenvalues with respect to the uniform probability measure on \(\QQ\). Let \(1\leq M\leq n\), and suppose that
  \[
    n\lambda_M^2\geq64\kappa_0^2\log(2M).
  \]
  Then there exist \(X_n\in(\QQ)^n\) and \(Q\in\R^{n\times M}\) such that
  \[
    Q\tran Q=I_M\spaced{and}Q\tran G_k(X_n)Q\succeq\frac{n\lambda_M}{2}I_M.
  \]
\end{proposition}

\begin{proof}
  By Mercer's theorem \citep[Theorem~4.49]{steinwart2008support}, there are continuous functions \((\phi_j)_{j\geq1}\), orthonormal in \(L^2(\QQ)\), such that
  \[
    k(x,y)=\sum_{j\geq1}\lambda_j\phi_j(x)\phi_j(y),
  \]
  with absolute and uniform convergence. Define
  \[
    u(x):=\roundb{\lambda_1^{1/2}\phi_1(x),\dots,\lambda_M^{1/2}\phi_M(x)}\tran,\qquad \Lambda:=\diag\roundb{\lambda_1,\dots,\lambda_M}.
  \]
  If \(X\) is uniform on \(\QQ\), then \(\E[u(X)u(X)\tran]=\Lambda\). Moreover,
  \[
    \norm{u(x)}^2=\sum_{j=1}^M\lambda_j\phi_j(x)^2\leq k(x,x)\leq\kappa_0, \spaced{and}
    \lambda_1\leq\sum_{j\geq1}\lambda_j=\int_\QQ k(x,x)\dif x\leq\kappa_0.
  \]

  Let \(X_1,\dots,X_n\) be independent uniform points in \(\QQ\), and let \(V\in\R^{n\times M}\) have rows \(u(X_1)\tran,\dots,u(X_n)\tran\). Set \(\widehat\Sigma:=n^{-1}V\tran V\). The matrices \(A_i:=u(X_i)u(X_i)\tran-\Lambda\) are independent, mean zero, and satisfy \(-\kappa_0I_M\preceq A_i\preceq\kappa_0I_M\). The matrix Hoeffding inequality \citep[Theorem~1.3]{tropp2012user} gives
  \[
    \P{\norm{\widehat\Sigma-\Lambda}_{\op}\geq\lambda_M/2}\leq2M\exp\curlyb[\Big]{-\frac{n\lambda_M^2}{32\kappa_0^2}}<1.
  \]
  Hence there is a deterministic realization \(X_n=(x_1,\dots,x_n)\in(\QQ)^n\) for which
  \[
    \widehat\Sigma\succeq\Lambda-\frac{\lambda_M}{2}I_M\succeq\frac{\lambda_M}{2}I_M.
  \]
  In particular, \(V\tran V\) is invertible. Define \(Q:=V(V\tran V)^{-1/2}\). Then \(Q\tran Q=I_M\). The truncated Mercer kernel \(k_M(x,y):=\sum_{j=1}^M\lambda_j\phi_j(x)\phi_j(y)\) satisfies \(k-k_M\succeq0\), and \(G_{k_M}(X_n)=VV\tran\). Therefore
  \[
    Q\tran G_k(X_n)Q\succeq Q\tran VV\tran Q=V\tran V=n\widehat\Sigma\succeq\frac{n\lambda_M}{2}I_M. \qedhere
  \]
\end{proof}

\subsection{Super-exponential spectral decay}\label{app:sup-exp-lb}

The super-exponential proof requires the Mercer eigenvalue decay of a product witness kernel. We first relate the product spectrum to the one-dimensional eigenvalues.

\begin{lemma}
\label{lem:eigenvalue-product-count}
  Let $(\lambda_m)_{m\in\Np}$ be a nonincreasing sequence of positive numbers such that
  \[
    -\log\lambda_m\asymp m\log m
  \]
  for all sufficiently large $m$. Let $(\Lambda_m)_{m\in\Np}$ be the nonincreasing rearrangement, counting multiplicity, of the products $\prod_{j=1}^d\lambda_{z_j}$, $z \in \Np^d$. Then
  \[
    -\log\Lambda_m\asymp m^{1/d}\log m
  \]
  for all sufficiently large $m$.
\end{lemma}

\begin{proof}
  Multiplying every $\lambda_m$ by a fixed positive constant changes $-\log\Lambda_m$ by an additive constant, so we may assume that $\lambda_m\leq e^{-1}$ for every $m$. Set $a_m:=-\log\lambda_m$ and $N_1(t):=\card\curlyb{m\in\Np:a_m\leq t}$. Since $a_m\asymp m\log m$ for all sufficiently large $m$, we have $N_1(t)\asymp t/\log t$ for all sufficiently large $t$. Indeed, this follows from monotonicity of $(a_m)$ and the relation $(t/\log t)\log(t/\log t)\asymp t$.

  For $z\in\Np^d$, set $b(z):=\sum_{j=1}^d a_{z_j}$ and $N_d(t):=\card\curlyb{z\in\Np^d:b(z)\leq t}$. Since $a_m\geq0$,
  \[
    \roundb{N_1(t/d)}^d\leq N_d(t)\leq\roundb{N_1(t)}^d,
  \]
  and hence $N_d(t)\asymp t^d(\log t)^{-d}$ for all sufficiently large $t$.

  Set $c_m:=-\log\Lambda_m$ and $r_m:=m^{1/d}\log m$. Then $(c_m)_{m\in\Np}$ is the nondecreasing rearrangement of $(b(z))_{z\in\Np^d}$, counting multiplicity, and
  \[
    c_m=\inf\curlyb{t>0:N_d(t)\geq m}.
  \]
  For every fixed $A>0$, $N_d(Ar_m)\asymp A^dm$ for all sufficiently large $m$. We may therefore choose $A_->0$ sufficiently small and $A_+>0$ sufficiently large that
  \[
    N_d(A_-r_m)<m\leq N_d(A_+r_m)
  \]
  for all sufficiently large $m$. It follows that $A_-r_m<c_m\leq A_+r_m$, proving the result.
\end{proof}
We next obtain the required one-dimensional eigenvalue decay from the asymptotics of \citet{widom1964asymptotic}.

\begin{lemma}
\label{lem:analytic-eigenvalues}
  Let $\gamma>1$ and $\tau>0$, and let $k_\gamma$ be the restriction to $\QQ\times\QQ$ of the stationary kernel on $\Rd$ with spectral density
  \[
    s_\gamma(\xi):=\exp\curlyb{-\tau\sum_{j=1}^d\abs{\xi_j}^\gamma}=\prod_{j=1}^de^{-\tau\abs{\xi_j}^\gamma}.
  \]
  Then its Mercer eigenvalues $(\lambda_m)_{m\in\Np}$ with respect to the uniform probability measure on $\QQ$ satisfy
  \[
    -\log\lambda_m\asymp m^{1/d}\log m
  \]
  for all sufficiently large $m$.
\end{lemma}

\begin{proof}
  Let $(\mu_m)_{m\in\Np}$ denote the Mercer eigenvalues, with respect to the uniform probability measure on $[0,1]$, of the one-dimensional kernel generated by $\xi\mapsto e^{-\tau\abs{\xi}^\gamma}$. Set \(\Gamma(\sigma):=\tau\abs{\sigma}^\gamma\). Since \(\gamma>1\), \(\Gamma\) is even, convex for large \(\sigma\), and satisfies \(\Gamma(\sigma)/\sigma\to\infty\). Hence \citet[Theorem~III]{widom1964asymptotic} gives
  \[
    \log\mu_m\sim-\Gamma(\sigma_m),
  \]
  where \(\sigma_m\) is the unique solution of
  \[
    \tau\sigma^\gamma=2m\log\roundb{m/\sigma},\qquad 0<\sigma<m.
  \]
  We claim that \(\sigma_m\asymp(m\log m)^{1/\gamma}\). Indeed, for \(\sigma=A(m\log m)^{1/\gamma}\),
  \[
    \tau\sigma^\gamma=\tau A^\gamma m\log m,
  \]
  while
  \[
    2m\log\roundb{m/\sigma}=2m\roundb{\roundb{1-\frac1\gamma}\log m-\frac1\gamma\log\log m-\log A}.
  \]
  For \(A>0\) sufficiently small, the left-hand side is smaller than the right-hand side for all sufficiently large \(m\), while for \(A>0\) sufficiently large the reverse inequality holds. Since the two sides are respectively increasing and decreasing in \(\sigma\), the claim follows. Consequently, \(\Gamma(\sigma_m)\asymp m\log m\), and therefore \(-\log\mu_m\asymp m\log m\).

  Since $s_\gamma$ factorises coordinatewise, $k_\gamma$ is a product of one-dimensional kernels, and its Mercer eigenvalues are the products $\prod_{j=1}^d\mu_{z_j}$, $z\in\Np^d$, counting multiplicity. Applying \cref{lem:eigenvalue-product-count} gives the result.
\end{proof}

We now combine the Mercer eigenvalue estimate with \cref{prop:mercer-directions} to construct a design whose Gram matrix is bounded below on a subspace of the required dimension.

\ThmSupExpLB*

\begin{proof}
  Write \(a:=n/\rho\), \(L:=L(a)\), and \(\ell:=\log L\). Choose \(a_0\geq1\), depending only on \(d,\gamma,\tau,c_0\), sufficiently large for the estimates below. We first assume that \(a\geq a_0\).

  By equivalence of norms on \(\Rd\), there exists \(\tau_0>0\) such that
  \[
    e^{-\tau\norm{\xi}^{\gamma}}\geq e^{-\tau_0\norm{\xi}_\gamma^\gamma},\qquad \xi\in\Rd.
  \]
  Let \(K_\gamma\) be the stationary kernel with spectral density \(S_\gamma(\xi):=e^{-\tau_0\norm{\xi}_\gamma^\gamma}\). By \cref{lem:spectral-domination},
  \[
    G_k(X)\succeq c_0G_{K_\gamma}(X)
  \]
  for every finite design \(X\).

  Let \((\lambda_m)_{m\in\Np}\) be the Mercer eigenvalues of \(K_\gamma\) with respect to the uniform probability measure on \(\QQ\). By \cref{lem:analytic-eigenvalues},
  \[
    \lambda_m\geq\exp\curlyb{-Cm^{1/d}\log(e+m)}
  \]
  for all sufficiently large \(m\). Fix \(A>0\), to be chosen below, and set
  \[
    M:=\floor[\big]{A\roundb{L/\ell}^d}.
  \]
  Increasing \(a_0\) if necessary gives \(M\leq n\) and places \(M\) in the range of the preceding eigenvalue bound. Moreover,
  \[
    M^{1/d}\log(e+M)\lesssim A^{1/d}L.
  \]
  Choose \(A>0\) sufficiently small that \(\lambda_M\geq e^{-L/8}\). Increasing \(a_0\) once more gives \(\lambda_M\geq a^{-1/4}\).

  Let \(\kappa_0:=\sup_{x\in\QQ}K_\gamma(x,x)\). Since \(\rho\geq1\), $n\lambda_M^2\geq na^{-1/2}=\rho a^{1/2}\geq a^{1/2}$, while \(\log(2M)\lesssim\log L\). Increasing \(a_0\) once more gives
  \[
    n\lambda_M^2\geq64\kappa_0^2\log(2M).
  \]
  By \cref{prop:mercer-directions}, there exist \(X_n\in(\QQ)^n\) and \(Q\in\R^{n\times M}\) such that \(Q\tran Q=I_M\) and
  \[
    Q\tran G_{K_\gamma}(X_n)Q\succeq\frac{n\lambda_M}{2}I_M.
  \]
  Spectral domination gives
  \[
    Q\tran G_k(X_n)Q\succeq\frac{c_0n\lambda_M}{2}I_M=\rho\theta I_M,\qquad \theta:=\frac{c_0a\lambda_M}{2}\gtrsim a^{3/4}.
  \]
  Applying \cref{prop:resolved-directions},
  \[
    \Effd_n^k(\QQ;\rho)\geq M\frac{\theta}{1+\theta}\spaced{and}\Gamma_n^k(\QQ;\rho)\geq M\log(1+\theta).
  \]
  Since \(M\asymp L^d\ell^{-d}\), \(\theta/(1+\theta)\gtrsim1\), and \(\log(1+\theta)\gtrsim L\), this gives
  \[
    \Effd_n^k(\QQ;\rho)\gtrsim L(a)^d\roundb{\log L(a)}^{-d}\spaced{and}\Gamma_n^k(\QQ;\rho)\gtrsim L(a)^{d+1}\roundb{\log L(a)}^{-d}
  \]
  when \(a\geq a_0\). Finally, \cref{rem:extend-lower-bounds} extends these bounds from \(a\in[a_0,n]\) to \(a\in[1,n]\).
\end{proof}
\end{document}